\documentclass[12pt,a4paper]{article}
\usepackage{amsmath,amsfonts,amssymb,amsthm,graphicx}
\usepackage[update]{epstopdf}
\usepackage[arrow, matrix, curve]{xy}
\usepackage{tikz}
\newtheorem{theorem}{Theorem}
\newtheorem{lemma}{Lemma}
\newtheorem{definition}{Definition}

\newtheorem{remark}{Remark}
\newtheorem{example}{Example}
\usepackage[utf8]{inputenc}

\newcommand\set[1]{\left\{#1\right\}}
\newcommand\abs[1]{\left\lvert#1\right\rvert}
\newcommand\norm[1]{\left\lVert#1\right\rVert}
\newcommand\map[3]{#1:#2\to#3}
\newcommand\NRp{\mathbb{R}_{\ge0}}
\newcommand\NRpp{\mathbb{R}_{>0}}

\title{Gromov meets  Phylogenetics --- new Animals for the Zoo of Biocomputable Metrics on Tree Space}
\author{Volkmar Liebscher}
\date{\day22\month04\year2015\relax\today,13:14:54
}

\begin{document} 
\maketitle
\begin{abstract}
  We present a new class of metrics for unrooted phylogenetic $X$-trees derived from  the Gromov-Hausdorff distance for (compact) metric spaces. These metrics can be efficiently computed by linear or quadratic programming. They are robust under NNI-operations, too.  The local behavior of the metrics shows that they are different from any formerly introduced metrics. The performance of the metrics is briefly analised on random weighted and unweighted trees as well as random caterpillars.   
\end{abstract}

\section{Introduction}

The idea for this paper came from a talk of Michelle Kendall at the Portobello conference 2015, see \cite{Ken15}. Basically, she postulated, that the  biological information is essentially  encoded in the collection of   distances between the MRCA of two taxa and the root. If the trees were ultrametric, we could equivalently just collect the distances between all pairs of taxa.  That leads to our rationale:
\begin{center}
  Instead of trees we compare the induced  metric spaces.
\end{center}
This approach is feasible  since by the work of Buneman  \cite{Bun71,Bun74}, see also \cite{Zar65} for the unweighted case, we can identify tree-induced metrics among all metrics by the famous four point conditions. 

In fact, this rationale must have been  behind  the invention of the $\ell^1$ and $\ell^2$ path difference  distances  \cite{WC71,PH85} already. Below we invent also an $\ell^\infty$ version of that metrics, too.

For (compact) metric spaces there is the well-known Gromov-Hausdorff distance
\begin{equation}\label{eq:gromov}
  D^{GH}((X,d),(X',d'))=\inf_{\varphi,\varphi'}\rho^H(\varphi(X),\varphi'(X'))
\end{equation}
where the infimum is taken over all isometric embeddings of $X,X'$ into a common metric space, and $\rho^H$ is the Hausdorff metric on the compacts of that space.

By our rationale, this definition induces  a metric on the space of all weighted trees. But, we cannot distinguish  trees which yield isomorphic metric spaces, i.e. with permuted labels.  Since  our aim is  to compare trees with the same taxon sets we have to adapt the metric (\ref{eq:gromov}) to our situation.  That makes the definition more complicated (see section 2) since we have to match the leaf labels, but the idea of embeddings remains. Fortunately, our metric becomes efficiently computable only this way.  Simply, we must   substitute the Hausdorff metric   in (\ref{eq:gromov}). Since there are several reasonable candidates for that, we derive  even three different metrics. In all these cases,   the value of the metric is the solution of a  linear or quadratic program.

 Clearly, our approach is more general  and abstract than other definitions of phylogenetic metrics to be discussed soon.  Those are using much more the internal structure of trees. Usually, more abstract approaches have more potential to generalise and to adapt to special situations. Still, this has to be worked out in the present situation. 

For mathematical reasons, it is very convenient to include also semimetrics on the taxon set  in the definition. This situation may occur in phylogeny if we do not resolve the topology by all singleton splits, see for instance \cite{RF81}.     

\medskip
What about other phylogenetic metrics? The simplest one, though  not the oldest   one, seems to be the Robinson-Foulds distance \cite{RF79,RF81}. That one is easy and efficiently to compute  in linear time \cite{Day85} or even in sublinear approximation \cite{PGM07}. But, it has no much power in discriminating trees, since a lot of trees with similar biological meaning have  distance equal to the diameter of the unweighted tree space. Much nearer to biology  seems to be a variant of the Robinson-Foulds  distance,  the weighted matching distance. It captures similarity of splits which entails a lot of biology and is still computable in subcubic time \cite{BM12,LRM12}.

A quite natural, biology adapted  way of capturing tree similarity  is provided by  the tree rearrangement metrics. There are different basic transformations giving rise to the  NNI-distance \cite{Rob71}, SPR-distance and TBR-distance. Unfortunately, computation of those distances is NP-hard and only feasible for small trees \cite{Das+97,AS01,BJ10}. Some fixed parameter approach to compute the (rooted) SPR, e.g,  was done in \cite{WBZ15}. Even more at the heart of evolution  is the maximum parsimony distance \cite{FK14}.  Still it is NP-hard to compute that distance, even over binary unweigthed phylogenetic trees \cite{FK14,FK14b}.

A good alternative to the tree rearrangement  metrics  is  the quartet distance \cite{Est85}. It is much more  biologically plausible than the Robinson-Foulds distance and also efficiently computable \cite{Bro01}.

For weighted phylogenetic trees there is the  euclidean type (geodesic) distance on tree space introduced by \cite{BHV01}. The crucial observation was that in a natural way tree space is a category zero ($\mathrm{CAT}(0)$) space (or space of nonpositive curvature) introduced by Gromov. Essentially this property implies uniqueness of geodesics.   It was an open problem for some years how to compute the geodesic distance on tree space efficiently. Yet, by \cite{OP11} we have a polynomial time algorithm now.  The $\mathrm{CAT}(0)$ idea was used again in  \cite{DG15} to develop metrics   for ultrametric spaces. Again, efficient computation of the geodesics is possible for at least one of the metrics.  As observed in that work,  different, but natural, parametrisations may yield different geodesics.

Recently, \cite{Ken15} returned back to the idea of \cite{WC71}, \cite{PH85} and \cite{BHV01} in application to weighted rooted trees, considering  all distances of MRCAs of pairs of taxa to the root. She also proposes to weight different MRCAs  by their depth respective  the root. That idea may be similar to the weighted matching distance \cite{BM12,LRM12}. 

A good review about recent developments in polynomial time computable metrics on unweighted phylogenetic trees is contained in \cite{BGW12}. There also    complete \texttt{java} implementations are provided. For simplicity, we implemented our metrics in \texttt{R} first. 

\medskip
After having this short overview over this situation, we would like to introduce  the notion of a \emph{biocomputable} metric. That should be a metric on phylogenetic tree space which is computable in polynomial time and which is able to capture biological similarity. Preferably, it should be also defined for weighted phylogenetic trees. So, let's see how  Gromovs idea of joint embeddings helps to reach  that goal \dots

\section{Definition}
For a  set $X$ denote by $M(X)$ the set of all semimetrics on $X$, i.e. all $\map \rho{X\times X}\NRp$ such that for all $x,y,z\in X$ $\rho(x,x)=0$, $\rho(x,y)=\rho(y,x)$ and $\rho(x,y)\le\rho(x,z)+\rho(z,y)$. Frequently, we describe such a semimetrics in an equivalent fashion by $\map\rho{\binom X2}\NRp$ where $\binom X2=\set{\set{x,y}:x,y\in X,x\ne y}$. Accordingly, $M_{>0}(X)$ denote the set of all metrics on $X$. Further, let $\mathcal M=\set{(X,\rho):\#X<\infty,\rho\in M(X)}$ denote the set of all finite semimetric spaces. Isometries $\map\varphi {(X,\rho)}{(Y,\rho')}$ preserve the semimetrics, i.e. for all $x,y\in X$ $ \rho(x,y)=\rho'(\varphi(x),\varphi(y))$.  

Frequently we need identical copies of our taxon set $X$. Under slight abuse of notation, we will denote them  $X'=\set{x':x\in X}$ and $X''=\set{x'':x\in X}$.  

\begin{definition}
  Let  $X$ be  a finite set. Then we define three functions $D_1,D_2,D_\infty$ on $M(X)\times M(X) $ by

  \begin{eqnarray*}
    D_1(\rho,\rho')&=&\inf_{Y,\varphi,\psi}\sum_{x\in X}\bar d(\varphi(x),\psi(x))
\\
    D_2(\rho,\rho')^2&=&\inf_{Y,\varphi,\psi}\sum_{x\in X}\bar d(\varphi(x),\psi(x))^2\\
    D_\infty(\rho,\rho')&=&\inf_{Y,\varphi,\psi}\max_{x\in X}\bar d(\varphi(x),\psi(x))
  \end{eqnarray*}
where the infimum is taken over all     
$(Y,\bar d)\in\mathcal M$ and all isometries  
$\map\varphi{(X,\rho)}{(Y,\bar d)}$, $\map\psi{(X,\rho')}{(Y,\bar d)}$.
\end{definition}

\begin{remark}
  $D_\infty$ is nearest to the Gromov-Hausdorff distance, which  we should implement via
  \begin{equation}\label{eq:Gromovhere}
    D_{GH}(\rho,\rho')=\inf_{Y,\varphi,\psi}\max_{x\in X}\min_{y\in X}\bar d(\varphi(x),\psi(y)).
  \end{equation}
On the other hand, we think that the  $\ell^1$-like metric $D_1$  is kind of natural for trees.  The   euclidean geometry  which is  the basis of  $D_2$ might be good for having unique geodesics. This feature is very convenient and at the heart of the proposals of \cite{BHV01} and \cite{DG15}.
\end{remark}
Let us simplify the optimisation problems present in the definitions of $D_i$ a bit.  In fact, it is enough to have just one model space $Y$. 
For $\rho,\rho'\in M(X)$ define the space $E(\rho,\rho')$ of their extensions
\begin{displaymath}
  E(\rho,\rho')=\set{\bar d\in M(X\cup X'): \forall x,y\in X:\bar d(x,y)=\rho(x,y),\bar d(x',y')=\rho'(x,y)}.
\end{displaymath}
Further, $\norm\cdot_i$ denotes the usual $\ell^i-$norm on $\mathbb{R}^X$. 

\begin{lemma}\label{lem:DibyextensionXX'}
 For $i=1,2,\infty$ 
  \begin{equation}\label{eq:DibyextensionXX'}
     D_i(\rho,\rho')=\inf_{\bar d\in E(\rho,\rho')}\norm{(\bar d(x,x'))_{x\in X}}_i
  \end{equation}
\end{lemma}
\begin{proof}
  Note that  $\le$ holds trivially.

On the other side, for $(Y,\tilde d)\in\mathcal M$  and isometries 
$\map\varphi{(X,\rho)}{(Y,\tilde d)}$,  $\map\psi{(X,\rho')}{(Y,\tilde d)}$ define $\map{ \bar d}{\binom{X\cup X'}2}\NRp$ by
\begin{eqnarray*}
  \bar d(x,y)&=&\rho(x,y)=\tilde d(\varphi(x),\varphi(y))\\
  \bar d(x',y')&=&\rho'(x,y)=\tilde d(\psi(x),\psi(y))\\
  \bar d(x,y')&=&\tilde d(\varphi(x),\psi(y))
\end{eqnarray*}
for all $x,y\in X$. Now $\tilde d\in M(Y)$ implies $\bar d\in M(X\cup X')$. The $\ge$ in (\ref{eq:DibyextensionXX'}) follows now from
\begin{displaymath}
  \norm{(\bar d(x,x'))_{x\in X}}_i=\norm{(\tilde d(\varphi(x),\psi(x)))_{x\in X}}_i
\end{displaymath}
\end{proof}
Observe that the previous lemma is at the heart of the computation of the distances since that amounts just to the minimization of a convex function over the convex set $E(\rho,\rho')$.
\begin{lemma}\label{lem:di*}
  For $i=1,2,\infty$ there exists a $d^*_i\in E(\rho,\rho') \subset M(X\cup X')$ such that
  \begin{displaymath}
    D_i(\rho,\rho')=\norm{(d_i^*(x,x'))_{x\in X}}_i
  \end{displaymath}
\end{lemma}
\begin{proof}
  Clearly, the sublevel sets of the convex function $\norm\cdot_i$ are compact on the convex space $E(\rho,\rho')$. Thus there must exist a minimal point  of that function. 
\end{proof}

\begin{theorem}\label{th:metrics}
  $D_i$, $i=1,2,\infty$ are complete metrics on $M(X)$. 
\end{theorem}
\begin{proof}
  Symmetry is clear.

If $D_i(\rho,\rho')=0$ choose $d_i^*\in  E(\rho,\rho') $ according to  the previous lemma.  Obviously, we obtain $d_i^*(x,x')=0$ for all $x\in X$.  The triangle inequality implies  for all $x,y\in X$
\begin{displaymath}
  \rho(x,y)=d^*_i(x,y)=d_i^*(x',y')=\rho'(x,y)
\end{displaymath}
such that $\rho=\rho'$.

Now let there be $\rho,\rho',\rho''\in M(X)$ and $i$ arbitrary. Using again the above lemma we choose $d_1\in M(X\cup X')$ extending  $\rho,\rho'$ and $d_2\in M(X'\cup X'')$ extending $\rho',\rho''$ such that 
\begin{eqnarray*}
  D_i(\rho,\rho')&=& \norm{(d_1(x,x'))_{x\in X}}_i\\
  D_i(\rho',\rho'')&=& \norm{(d_2(x',x''))_{x\in X}}_i
\end{eqnarray*}
 Then we find, following \cite{Cri08} or   Lemma \ref{lem:biextension},  some $d\in M(X\cup X'\cup X'')$ extending both $d_1,d_2$:
 \begin{displaymath}
   d|_{\binom{X\cup X'}2}=d_1\qquad\text{and}\qquad  d|_{\binom{X'\cup X''}2}=d_2.
 \end{displaymath}
We see now
\begin{eqnarray*}
  D_i(\rho,\rho'')&\le&  \norm{(d(x,x''))_{x\in X}}_i\le \norm{(d(x,x')+d(x',x''))_{x\in X}}_i\\
&\le& \norm{(d(x,x'))_{x\in X}}_i+\norm{(d(x',x''))_{x\in X}}_i\\
&=&\norm{(d_1(x,x'))_{x\in X}}_i+\norm{(d_2(x',x''))_{x\in X}}_i\\
&=& D_i(\rho,\rho')+ D_i(\rho',\rho'').
\end{eqnarray*}

Completeness will be proved later in Lemma \ref{lem:complete}.    
\end{proof}

As already said in the introduction, we are mainly interested in metrics on tree space.
Let $G=(V,E,q)$ be a weighted connected graph, i.e. $E\subseteq \binom V2$ and $\map qE\NRp$. The we define the induced semimetric on $V$ by
\begin{equation}\label{eq:dqG}
  d_G^q(x,y)=\inf\set{\mathrm{len}(p):p\mbox{~path from $x$ to $y$~}}
\end{equation}
As usual,
\begin{displaymath}
  \mathrm{len}(x_0x_1\dots x_m)=\sum_{i=1}^mq(\set{x_{i-1},x_i})
\end{displaymath}
is here the length of the path $(x_0x_1\dots x_m)$. 
For unweighted graphs $(V,E)$ we choose $q(\set{x,y})=1$ for all $\set{x,y}\in E$. 

 So let the tree space $T(X)$  be the set of  all weighted unrooted  generalised phylogenetic $X$-trees. A  weighted unrooted  generalised phylogenetic $X$-tree is a quadruple $(V,E,q,\mu)$, where $\map \lambda XV$ is  a (not necessarily injective) map such that $(V,E)$ is the minimal  tree containing  $\mu(X)$ and $\map qE\NRpp$ is a weight function.   Phylogenetic $X$-trees without weights are included by given all edges after contraction a weight of 1 and by  requiring $\mu$ to be injective. The corresponding  subspace will be denoted $T_1(X)$.  The set of  binary (bifurcating) phylogenetic  $X$-trees is  denoted $T_1^2(X)$. Now   we define for $\tau,\tau'\in T(X)$ under abuse of notation
\begin{displaymath}
  D_i(\tau,\tau')=D_i(d_\tau|_{\binom X2},d_{\tau'}|_{\binom X2})
\end{displaymath}
where $\rho\in M(X)$ is induced by the tree $\tau_1$ and $\rho'$ by $\tau_2$ via (\ref{eq:dqG}).   Again, all three are metrics on tree space. This can be seen from the following result, provided in essence by \cite{Bun71}.
\begin{lemma}\label{lem:4point}
  For $\rho \in M(X)$ there exists an unrooted generalised phylogenetic $X-$tree $\tau\in T(X)$ with $\rho=d_\tau|_{\binom X2}$ if and only if for for all $x,y,z,w\in X$ the four point condition
  \begin{equation}
    \label{eq:4point}
    \rho(x,y)+\rho(z,w)\le \max(\rho(x,z)+\rho(y,w), \rho(x,w)+\rho(y,z))
  \end{equation}
is fulfilled.
\end{lemma}
\begin{proof}
  Identifying points $x,y\in X$ with $\rho(x,y)=0$ we can assume that $\rho$ is a metric. That (\ref{eq:4point}) is necessary and sufficient now for the existence of $\tau$ was shown in \cite{Bun71}. The splits of $\tau$ are identified by situations where  in (\ref{eq:4point}) strict inequality holds. Minimality of the vertex set of $\tau$ (according to definition) implies that different edges in $\tau$ induce different splits. The weight of the edge corresponding to a split by   (\ref{eq:4point}) computes directly from the difference of the right and the left hand side in (\ref{eq:4point}). Thus $\tau\in T(X)$ is uniquely determined. 
\end{proof}
Let us compute some example. 
\begin{example}\label{ex:1}
  We want to compare for  $X=\set{A,B,C,D}$ the two unweigthed $X-$trees
  \begin{displaymath}
\tau=\begin{tikzpicture}[xscale=0.5,yscale=0.2, baseline =10 ]
\node (A) at (0,4) [rectangle,rounded corners,draw=black] {$A$};
\node (B) at (0,0) [rectangle,rounded corners,draw=black] {$B$};
\node (C) at (2,2)  {$\bullet$};
\node (D) at (4,2)  {$\bullet$};
\node (E) at (6,0) [rectangle,rounded corners,draw=black] {$C$};
\node (F) at (6,4) [rectangle,rounded corners,draw=black] {$D$};
\draw[very thick] (A) to (C);
\draw[very thick] (B) to (C);
\draw[very thick] (C) to (D);
\draw[very thick] (D) to (E);
\draw[very thick] (D) to (F);
\end{tikzpicture}
\qquad \mbox{and}\qquad \tau'=\begin{tikzpicture}[xscale=0.5,yscale=0.2, baseline =10 ]
\node (A) at (0,4) [rectangle,rounded corners,draw=black] {$A$};
\node (B) at (0,0) [rectangle,rounded corners,draw=black] {$C$};
\node (C) at (2,2)  {$\bullet$};
\node (D) at (4,2)  {$\bullet$};
\node (E) at (6,0) [rectangle,rounded corners,draw=black] {$B$};
\node (F) at (6,4) [rectangle,rounded corners,draw=black] {$D$};
\draw[very thick] (A) to (C);
\draw[very thick] (B) to (C);
\draw[very thick] (C) to (D);
\draw[very thick] (D) to (E);
\draw[very thick] (D) to (F);
\end{tikzpicture}
\end{displaymath}
 with corresponding distances $\rho,\rho'$. 

We want to  derive possible extensions of  $\rho,\rho'$ by  verifying that for some  $\delta_A,\delta_B,\delta_C,\delta_D\ge0$ the graph distances on the weighted graph 
\begin{displaymath}
  G=\begin{tikzpicture}[baseline=25,xscale=0.5,yscale=0.4 ]
\node (A) at (3,6) [rectangle,rounded corners,draw=black] {$A$};
\node (B) at (3,4) [rectangle,rounded corners,draw=black] {$B$};
\node (E) at (0,5)  {$\bullet$};
\node (F) at (0,1)  {$\bullet$};
\node (C) at (3,2) [rectangle,rounded corners,draw=black] {$C$};
\node (D) at (3,0) [rectangle,rounded corners,draw=black] {$D$};
\draw[very thick] (E) to (A);
\draw[very thick] (E) to (B);
\draw[very thick] (F) to (C);
\draw[very thick] (F) to (D);
\draw[very thick] (E) to (F);
\node (As) at (8,6) [rectangle,rounded corners,draw=black] {$A'$};
\node (Cs) at (8,4) [rectangle,rounded corners,draw=black] {$C'$};
\node (Es) at (11,5)  {$\bullet$};
\node (Fs) at (11,1)  {$\bullet$};
\node (Bs) at (8,2) [rectangle,rounded corners,draw=black] {$B'$};
\node (Ds) at (8,0) [rectangle,rounded corners,draw=black] {$D'$};
\draw[very thick] (Es) to (As);
\draw[very thick] (Es) to (Cs);
\draw[very thick] (Fs) to (Bs);
\draw[very thick] (Fs) to (Ds);
\draw[very thick] (Es) to (Fs);
\draw[dotted,thick] (A) to (As);
\draw[dotted,thick] (B) to (Bs);
\draw[dotted,thick] (C) to (Cs);
\draw[dotted,thick] (D) to (Ds);
\node [above] at (5.5,6) {$\delta_A$};
\node [above] at (6.5,3.5) {$\delta_C$};
\node [above] at (4.5,3.5) {$\delta_B$};
\node [above] at (5.5,0) {$\delta_D$};
\end{tikzpicture}
\end{displaymath}
reproduce  both   $\rho$ and $\rho'$. One obvious choice is $\delta_A=0,\delta_B=1,\delta_C=1,\delta_D=0$, i.e. 
\begin{displaymath}
  G=\begin{tikzpicture}[baseline=25,xscale=0.5,yscale=0.4 ]
%\node (A) at (3,6) [rectangle,rounded corners,draw=black] {$A$};
\node (B) at (3,4) [rectangle,rounded corners,draw=black] {$B$};
\node (E) at (0,5)  {$\bullet$};
\node (F) at (0,1)  {$\bullet$};
\node (C) at (3,2) [rectangle,rounded corners,draw=black] {$C$};
%\node (D) at (3,0) [rectangle,rounded corners,draw=black] {$D$};
\node (AAs) at (5.5,6) [rectangle,rounded corners,draw=black] {$A=A'$};
\node (Cs) at (8,4) [rectangle,rounded corners,draw=black] {$C'$};
\node (Es) at (11,5)  {$\bullet$};
\node (Fs) at (11,1)  {$\bullet$};
\node (Bs) at (8,2) [rectangle,rounded corners,draw=black] {$B'$};
\node (DDs) at (5.5,0) [rectangle,rounded corners,draw=black] {$D=D'$};
\draw[very thick] (E) to (AAs);
\draw[very thick] (E) to (B);
\draw[very thick] (F) to (C);
\draw[very thick] (F) to (DDs);
\draw[very thick] (E) to (F);
\draw[very thick] (Es) to (AAs);
\draw[very thick] (Es) to (Cs);
\draw[very thick] (Fs) to (Bs);
\draw[very thick] (Fs) to (DDs);
\draw[very thick] (Es) to (Fs);
%\draw[dotted] (A) to (As);
\draw[dotted,thick] (B) to (Bs);
\draw[dotted,thick] (C) to (Cs);
%\draw[dotted] (D) to (Ds);
%\node at (5.5,6) {$0$};
\node [above] at (6.5,3.5) {$1$};
\node [above] at (4.5,3.5) {$1$};
%\node at (5.5,0) {$0$};
\end{tikzpicture}
\end{displaymath}
is consistent. Obviously, we embedded now both $\tau$ and $\tau'$ into the metric space of the graph $G$.  We see $D_\infty\le1$, $D_2\le\sqrt2$ and $D_1\le2$. In fact equality holds, but this we can prove only later in Example \ref{ex:2}. 
\end{example}
  Additionally, we obtain

\begin{lemma}\label{lem:algebra}
 For $\lambda\ge0$, $i=1,2,\infty$, and $\rho_j\in M(X)$, $j=1,2,3,4$,  the following are true:
 \begin{enumerate}
 \item $D_i(\lambda\rho_1,\lambda\rho_2)=\lambda D_i(\rho_1,\rho_2)$.
 \item $D_i(\rho_1+\rho_3,\rho_2+\rho_4)\le  D_i(\rho_1,\rho_2)+ D_i(\rho_3,\rho_4)$.
 \item $D_i(\lambda\rho_1+(1-\lambda)\rho_2,\rho_3)\le\lambda D_i(\rho_1,\rho_3)+(1-\lambda)D_i(\rho_2,\rho_3)$.
 \end{enumerate}
\end{lemma}
\begin{proof}
  The first relation follows from $\lambda \bar d\in E(\lambda\rho_1,\lambda\rho_2)\iff\def\lambda{}\lambda \bar d\in E(\lambda\rho_1,\lambda\rho_2) $.

The second relation follows from $\bar d_1+\bar d_2\in E(\rho_1+\rho_3,\rho_2+\rho_4)$ for all $\bar d_1\in E(\rho_1,\rho_2)  $ and  $\bar d_2\in E(\rho_3,\rho_4)  $.

The third relation is just a consequence of the first two.
\end{proof}
\section{Efficient Computation}
Clearly, 
\begin{lemma}
  $D_1$ and $D_\infty$ can be computed solving a linear program. For the computation of  $D_2$ a quadratic program has to be solved.
\end{lemma}
\begin{proof}
  This follows immediately from Lemma \ref{lem:DibyextensionXX'}.
\end{proof}
So, we are sure that we can compute the distance in a computing time polynomially bounded in $n=\#X$ \cite{Karmakar}. In the naïve way, the linear (quadratic) program has the $n^2$ variables $\epsilon_{xy}=\bar d(x,y')$ and $\mathrm{O}(n^3)$ restrictions coming essentially from the triangle inequalities in triangles of the form $x,y,z'$ or similar. But we can do the computation more efficiently. The essential observation is that the objective function depends on the unknown values $(\bar d(x,x'))_{x\in X}$ only. The reformulation of  the constraints   is provided by the following theorem. It will be  proved later in section \ref{sec:metrext}. 

\begin{theorem}[quadrangle inequalities]\label{th:quadrin}
Let $\rho,\rho'\in M(X)$ and $(\delta_x)_{x\in X}\in \NRp^X$ be given. Then there exists a $\bar d\in M(X\cup X')$ with
\begin{eqnarray*}
  \bar d(x,y)&=&\rho(x,y)\qquad x,y\in X\\
  \bar d(x',y')&=&\rho'(x,y)\qquad x,y\in X\\
  \bar d(x,x')&=&\delta_x\qquad x\in X
\end{eqnarray*}
if and only if  for all $x\ne y\in X$ the following inequalities are fulfilled:
\begin{equation}
  \begin{array}[c]{rcl}
    \delta_x+\delta_y&\ge&\abs{\rho(x,y)-\rho'(x,y)}\\
    \abs{\delta_x-\delta_y}&\le&\rho(x,y)+\rho'(x,y)
  \end{array}\label{eq:quadr}
\end{equation}
\end{theorem}

Thus we have just $n$ variables  $\delta_{x}=\bar d(x,x')$ and $\mathrm{O}(n^2)$ constraints for each rectangle  $x,y,y',x'$  in the optimisation problems (\ref{eq:DibyextensionXX'}). 
Formally, $D_i(\rho,\rho')$ solves the program
\begin{equation}
  \label{eq:min1}
  \begin{array}[c]{*4c}
    \norm\delta_i&\to&\min\qquad\mbox{under}\\[1ex]
\delta_x&\ge&0&x\in X\\
    \delta_x+\delta_y&\ge&\abs{\rho(x,y)-\rho'(x,y)}&x\ne y\in X\\
    \abs{\delta_x-\delta_y}&\le&\rho(x,y)+\rho'(x,y)&x\ne y\in X\\
  \end{array}
\end{equation}

\begin{example}\label{ex:2}
  Let us continue Example \ref{ex:1}. Since $\rho(A,D)=\rho'(A,D)$,   we see from the upper parts of (\ref{eq:quadr})
  \begin{eqnarray*}
    \delta_A+\delta_B&\ge &1\\
    \delta_A+\delta_C&\ge &1\\
    \delta_B+\delta_D&\ge &1\\
    \delta_C+\delta_D&\ge &1
  \end{eqnarray*}
Consequently,
\begin{displaymath}
  D_1(\rho,\rho')\ge \delta_A+\delta_B+
    \delta_C+\delta_D\ge 2.
\end{displaymath}
We already saw that we can realise this minimum. The calculation of $D_\infty(\rho,\rho')=1$ was already done,  $D_2(\rho,\rho')=\sqrt 2$ is immediate. 
\end{example}

It is very interesting that the upper bounds on the differences are not used in the calculation. In fact, we could not observe any situation where they had to be  used to determine the minimum.  This can be  seen  also from the numerical results in  section \ref{sec:numeric}, especially Figure \ref{fig:gleich}. But, we are still lacking a proof that we may omit these constraints safely.  This leads us to the  definition of further distances $\tilde D_i(\rho,\rho')$ as solution of 
\begin{equation}
  \label{eq:min2}
  \begin{array}[c]{*4c}
    \norm\delta_i&\to&\min\qquad\qquad\mbox{\textrm{under}}\\
\delta_x&\ge&0&x\in X\\
    \delta_x+\delta_y&\ge&\abs{\rho(x,y)-\rho'(x,y)}&x\ne y\in X
  \end{array}
\end{equation}
with the obvious extension to tree space. 
\begin{lemma}\label{lem:tildemetrics}
  $\tilde D_i$ are metrics on $M(X)$ and $T(X)$, too. 
\end{lemma}
\begin{proof}
  Observe that exacly like for the problem (\ref{eq:min1}), also the minimum of (\ref{eq:min2}) is attained. 

Symmetry of the definition is clear. Further, $\tilde D_i(\rho,\rho')=0$ if and only if $\delta=0$ is feasible for the problem  (\ref{eq:min1}). That means $\rho(x,y)=\rho'(x,y)$ for all $\set{x,y}\in \binom X2$ and $\rho=\rho'$. 
 
For the proof of the triangle inequality choose  optimal solutions $\delta^1\in \NRp^X$ of (\ref{eq:min2}) and   $\delta^2\in \NRp^X$ of the version of (\ref{eq:min2}) for $\rho',\rho''$. 
We see for $\set{x,y}\in\binom X2$ that 
\begin{displaymath}
  \delta^1_x+\delta^2_x+\delta^1_y+\delta^2_y\ge\abs{\rho(x,y)-\rho'(x,y)}+\abs{\rho'(x,y)-\rho''(x,y)}\ge \abs{\rho(x,y)-\rho''(x,y)}
\end{displaymath}
such $\delta^1+\delta^2$ is feasible for the version of (\ref{eq:min2}) for $\rho,\rho''$. We obtain
\begin{displaymath}
  \tilde D_i(\rho,\rho'')\le \norm{\delta^1+\delta^2}_i\le \norm{\delta^1}_i+\norm{\delta^2}_i= \tilde D_i(\rho,\rho') +\tilde D_i(\rho',\rho'').
\end{displaymath}
This completes the proof.
\end{proof}
\begin{remark}
  Interestingly, there is a striking similarity between the feasible set of (\ref{eq:min2}) and the  tight span of a distance matrix introduced in \cite{Dre84}. Yet, $\abs{\rho-\rho'}$ is not a semimetric in general and we do not see a deeper connection at the moment.  
\end{remark}

\section{Comparison to other metrics}
First we compare our metrics to the pathwise difference metrics. Recall that those are defined by \cite{WC71,PH85}
\begin{displaymath}
  D^{PD}_i(\tau_1,\tau_2)=\norm{(\rho_{\tau_1}(x,y)-\rho_{\tau_2}(x,y))_{\set{x,y}\in\binom X2}}_i
\end{displaymath}

Interestingly, it  seems that  $D^{PD}_\infty$ was  not  used before. May be, we can immediately explain this.   Again we abbreviate $n=\#X$. 
\begin{theorem}\label{th:PDcomp}
For $\tau_1,\tau_2\in T(X)$ it holds
\begin{displaymath}
    \begin{array}[c]{*9c}
 D_1(\tau_1,\tau_2)  &\ge&D_2(\tau_1,\tau_2)&\ge& D_\infty(\tau_1,\tau_2)&\ge&\frac1{\sqrt n} D_2(\tau_1,\tau_2)&\ge&\frac1{n} D_1(\tau_1,\tau_2)\\
 \tilde D_1(\tau_1,\tau_2)  &\ge&\tilde D_2(\tau_1,\tau_2)&\ge& \tilde D_\infty(\tau_1,\tau_2)&\ge&\frac1{\sqrt n}\tilde  D_2(\tau_1,\tau_2)&\ge&\frac1{n} \tilde D_1(\tau_1,\tau_2)\\
   \end{array}
\end{displaymath}

  \begin{displaymath}
  \begin{array}[c]{*7c}
  \frac n2 D^{PD}_1(\tau_1,\tau_2)  &\ge&D_1(\tau_1,\tau_2)&\ge&\tilde D_1(\tau_1,\tau_2)&\ge&\frac1{n-1} D^{PD}_1(\tau_1,\tau_2)\\
  \frac{\sqrt n}2 D^{PD}_2(\tau_1,\tau_2)  &\ge&D_2(\tau_1,\tau_2)&\ge&\tilde D_2(\tau_1,\tau_2)&\ge&\sqrt{\frac2{n-1}} D^{PD}_2(\tau_1,\tau_2)\\
&&D_\infty (\tau_1,\tau_2)&=&\tilde D_\infty(\tau_1,\tau_2)&=&\frac12 D^{PD}_\infty(\tau_1,\tau_2)
  \end{array}
\end{displaymath}

\end{theorem}
\begin{proof}
The first relations are well-known for $\norm\cdot_i$ and translate directly. 

  For the second relation  we use the first inequality in (\ref{eq:quadr}). This gives us for all $x\ne y\in X$
\begin{eqnarray*}
   \delta_x+\delta_y&\ge&\abs{\rho(x,y)-\rho'(x,y)}\\
   \delta^2_x+\delta^2_y&\ge&\frac12( \delta_x+\delta_y)^2\ge\frac12\abs{\rho(x,y)-\rho'(x,y)}^2\\
\max\set{\delta_x:x\in X}&\ge&  \frac12( \delta_x+\delta_y)\ge\frac12\abs{\rho(x,y)-\rho'(x,y)}
  \end{eqnarray*}
Summing up the first or the second inequalities for all $\set{x,y}\in \binom X2$ gives the estimates for $i=1,2$. 

The $\ge$-estimate for $i=\infty$ follows by taking the maximum of the third inequality over all $\set{x,y}\in \binom X2$. On the other hand, setting
\begin{displaymath}
\delta_z=\max\set{\abs{\rho_{\tau_1}(x,y)-\rho_{\tau_2}(x,y)}:\set{x,y}\in\binom X2}
\end{displaymath}
 $z\in X$, (\ref{eq:quadr}) is clearly fulfilled and we obtain also the $\le$-estimate. 

The first  estimates yield the rest of the second  estimates and complete the proof. 
\end{proof}

By the same arguments as in Lemma \ref{lem:di*}, both (\ref{eq:min1}) and (\ref{eq:min2}) possess minimal points  $\delta^*\in \NRp^X$.  As a corollary of the last theorem we find a useful upper bound for the elements of these vectors:
\begin{lemma}\label{lem:upperbound}
  In the minimisation problems (\ref{eq:min1}) or (\ref{eq:min2}), we may restrict minimisation to $\delta\in \NRp^X$ which fulfil additionally
  \begin{displaymath}
    \delta_x\le 2 D_\infty(\rho,\rho')=D^{PD}_\infty(\rho,\rho').
  \end{displaymath}
E.g., the minimisation problems 
\begin{equation}
  \label{eq:min1b}
  \begin{array}[c]{*4c}
    \norm\delta_i&\to&\min\qquad\mbox{\textrm{under}}\\[1ex]
0\le \delta_x&\le&2D_\infty(\rho,\rho')&x\in X\\
    \delta_x+\delta_y&\ge&\abs{\rho(x,y)-\rho'(x,y)}&x\ne y\in X\\
    \abs{\delta_x-\delta_y}&\le&\rho(x,y)+\rho'(x,y)&x\ne y\in X\\
  \end{array}
\end{equation}
and 
\begin{equation}
  \label{eq:min2b}
  \begin{array}[c]{*4c}
    \norm\delta_i&\to&\min\qquad\mbox{\textrm{under}}\\[1ex]
0\le \delta_x&\le&2D_\infty(\rho,\rho')&x\in X\\
    \delta_x+\delta_y&\ge&\abs{\rho(x,y)-\rho'(x,y)}&x\ne y\in X
  \end{array}
\end{equation}
yield again $D_i(\rho,\rho')$ and $\tilde D_i(\rho,\rho')$ as minimal values, respectively.
\end{lemma}
\begin{proof}
  Define $\tilde \delta$ by $\tilde \delta_x=\min(\delta_x, 2 D_\infty(\rho,\rho'))$. By the above relation, $\tilde \delta$ is again in the feasible set of (\ref{eq:min1}) and (\ref{eq:min2}) respectively. Further,  $\norm{\tilde\delta}_i\le \norm\delta_i$ completes the proof. 
\end{proof}

\begin{lemma}\label{lem:complete}
  $M(X)$ and $T(X)$ are complete in each $D_i$, $i=1,2,\infty$.
\end{lemma}
\begin{proof}
  Clearly, $M(X)$ is complete w.r.t. $D_\infty^{PD}$. Since all metrics on $M(X)$ are equivalent by Theorem \ref{th:PDcomp},  the same should be true for $D_i$. On $T(X)$ we have to observe additionally, that $T(X)$ is closed since both sides of the four point conditions (\ref{eq:4point}) depend  continuously  on   the metric. Then Lemma \ref{lem:4point} implies completeness of $T(X)$. 
\end{proof}
To show that the new metrics are biologically meaningful, we show that they don't change much under an NNI (nearest neighbour interchange) operations. Such an operation is given by 
\begin{displaymath}
  \begin{tikzpicture}[xscale=0.5,yscale=0.2, baseline =10 ]
\node (A) at (0,4) [rectangle,draw=black] {$A$};
\node (B) at (0,0) [rectangle,draw=black] {$B$};
\node (C) at (2,2)  {$\bullet$};
\node (D) at (4,2)  {$\bullet$};
\node (E) at (6,0) [rectangle,draw=black] {$C$};
\node (F) at (6,4) [rectangle,draw=black] {$D$};
\draw[very thick] (A) to (C);
\draw[very thick] (B) to (C);
\draw[very thick] (C) to (D);
\draw[very thick] (D) to (E);
\draw[very thick] (D) to (F);
\end{tikzpicture}\qquad\longmapsto\qquad \begin{tikzpicture}[xscale=0.5,yscale=0.2, baseline =10 ]
\node (A) at (0,4) [rectangle,draw=black] {$A$};
\node (B) at (0,0) [rectangle,draw=black] {$C$};
\node (C) at (2,2)  {$\bullet$};
\node (D) at (4,2)  {$\bullet$};
\node (E) at (6,0) [rectangle,draw=black] {$B$};
\node (F) at (6,4) [rectangle,draw=black] {$D$};
\draw[very thick] (A) to (C);
\draw[very thick] (B) to (C);
\draw[very thick] (C) to (D);
\draw[very thick] (D) to (E);
\draw[very thick] (D) to (F);
\end{tikzpicture}
\end{displaymath}
or 
\begin{displaymath}
  \begin{tikzpicture}[xscale=0.5,yscale=0.2, baseline =10 ]
\node (A) at (0,4) [rectangle,draw=black] {$A$};
\node (B) at (0,0) [rectangle,draw=black] {$B$};
\node (C) at (2,2)  {$\bullet$};
\node (D) at (4,2)  {$\bullet$};
\node (E) at (6,0) [rectangle,draw=black] {$C$};
\node (F) at (6,4) [rectangle,draw=black] {$D$};
\draw[very thick] (A) to (C);
\draw[very thick] (B) to (C);
\draw[very thick] (C) to (D);
\draw[very thick] (D) to (E);
\draw[very thick] (D) to (F);
\end{tikzpicture}\qquad\longmapsto\qquad \begin{tikzpicture}[xscale=0.5,yscale=0.2, baseline =10 ]
\node (A) at (0,4) [rectangle, draw=black] {$A$};
\node (B) at (0,0) [rectangle,draw=black] {$D$};
\node (C) at (2,2)  {$\bullet$};
\node (D) at (4,2)  {$\bullet$};
\node (E) at (6,0) [rectangle,draw=black] {$C$};
\node (F) at (6,4) [rectangle,draw=black] {$B$};
\draw[very thick] (A) to (C);
\draw[very thick] (B) to (C);
\draw[very thick] (C) to (D);
\draw[very thick] (D) to (E);
\draw[very thick] (D) to (F);
\end{tikzpicture}
\end{displaymath}
where $A,B,C,D$ denote different subtrees. 
The minimal number of NNI operations to reach $\tau'\in T_1^2(X)$ from $\tau\in T_1^2(X)$ is the NNI-distance $D^{NNI}(\tau,\tau')$ \cite{Rob71}. 
\begin{theorem}
  Consider $\tau,\tau'\in T^2_1(X)$  which are away by one NNI operation.  Then
      \begin{displaymath}
  \begin{array}[c]{*5c}
 D_1(\tau,\tau')&\le& n\\
D_2(\tau,\tau')&\le&\sqrt n\\
D_\infty (\tau,\tau')&=&1
  \end{array}
\end{displaymath}
Especially,
\begin{displaymath}
  D^{NNI}(\tau,\tau')\ge D_\infty (\tau,\tau')\ge \frac1{\sqrt n}D_2(\tau,\tau')\ge \frac1n D_1(\tau,\tau').
\end{displaymath}

\end{theorem}
\begin{proof}
  Let be $\tau=\begin{tikzpicture}[xscale=0.5,yscale=0.2, baseline =10 ]
\node (A) at (0,4) [rectangle,draw=black] {$A$};
\node (B) at (0,0) [rectangle,draw=black] {$B$};
\node (C) at (2,2)  {$\bullet$};
\node (D) at (4,2)  {$\bullet$};
\node (E) at (6,0) [rectangle,draw=black] {$C$};
\node (F) at (6,4) [rectangle,draw=black] {$D$};
\draw[very thick] (A) to (C);
\draw[very thick] (B) to (C);
\draw[very thick] (C) to (D);
\draw[very thick] (D) to (E);
\draw[very thick] (D) to (F);
\end{tikzpicture}
$ and $\tau'=\begin{tikzpicture}[xscale=0.5,yscale=0.2, baseline =10 ]
\node (A) at (0,4) [rectangle,draw=black] {$A$};
\node (B) at (0,0) [rectangle,draw=black] {$C$};
\node (C) at (2,2)  {$\bullet$};
\node (D) at (4,2)  {$\bullet$};
\node (E) at (6,0) [rectangle,draw=black] {$B$};
\node (F) at (6,4) [rectangle,draw=black] {$D$};
\draw[very thick] (A) to (C);
\draw[very thick] (B) to (C);
\draw[very thick] (C) to (D);
\draw[very thick] (D) to (E);
\draw[very thick] (D) to (F);
\end{tikzpicture}
$ where $A,B,C,D$ are the four subtrees of $\tau,\tau'$ corresponding to a four-partition of $X$.

Then we observe the following structure of the matrix   $\Delta\in \NRp^{X^2}$, $\Delta_{x,y}=(\abs{\rho_{\tau}(x,y)-\rho_{\tau'}(x,y)})_{x,y\in X}$:
\begin{displaymath}
  \Delta=
  \begin{pmatrix}
    0&1&1&0\\
1&0&0&1\\
1&0&0&1\\
    0&1&1&0
  \end{pmatrix}
\end{displaymath}
or more precisely
\begin{displaymath}
  \Delta_{x,y}=\left\{
    \begin{array}[c]{cl}
      1&x\in A\cup D,y\in B\cup C\\
      1&y\in A\cup D,x\in B\cup C\\
0&\mbox{otherwise}
    \end{array}
\right.
\end{displaymath}

The estimates are now immediate from Theorem \ref{th:PDcomp}.
\end{proof}

\begin{remark}
 Similar  estimates could be done for the SPR-metrics. By  \cite{AS01} this has natural implications to the  TBR-metrics, too.  Further we see that the size of the $1-$neighbourhood of a tree $\tau\in T_1^2(X)$  in the $D_\infty-$metric is at least $n-3$. 
\end{remark}
How large are those bounds compared to the diameter of the space $T^2_1(X)$?
We have some crude estimates:
\begin{lemma}\label{lem:upperdiameter}
  For all $\tau_1,\tau_2\in T_1(X)$ it holds 
\begin{displaymath}
  \begin{array}[c]{*5c}
 D_1(\tau_1,\tau_2)&\le& n\cdot\frac{n-2}2\\
D_2(\tau_1,\tau_2)&\le&\sqrt n\cdot\frac{n-2}2\\
D_\infty (\tau_1,\tau_2)&\le &\frac{n-2}2
  \end{array}
\end{displaymath}
\end{lemma}
\begin{proof}
  $D_\infty (\tau_1,\tau_2)\le \frac{n-1}2$ follows immediately from $D^{PD}_\infty (\tau_1,\tau_2)\le {n-2}$ which holds since all paths in $\tau_1,\tau_2$ have at least one and at most $(n-1)$ edges. Theorem \ref{th:PDcomp} implies the other two inequalities and the estimate on the NNI-metric are immediate consequences of its definition. 
\end{proof}
Now we want to show that there are trees such that the distance between them is of  the same order in $n$. 
\begin{lemma}\label{lem:cater}
  Let us be given   $n=4m+1$ for some $m\in \mathbb{N}$, $m\ge1$, $X=\set{1,\dots,4m+1}$. Suppose  $\tau$ is the  unrooted caterpillar tree with cherries $\set{1,2}$ and $\set{4m,4m+1}$:
  \begin{displaymath}
    \tau=\begin{tikzpicture}[xscale=1.2,yscale=0.7, baseline =18 ]
\node (1) at (0,2) [rectangle,rounded corners,draw=black] {$1$};
\node (2) at (0,0) [rectangle,rounded corners,draw=black] {$2$};
\node (12s) at (1,1)  {$\bullet$};
\node (3) at (2,2) [rectangle,rounded corners,draw=black] {$3$};
\node (3s) at (2,1)  {$\bullet$};
\node (4) at (3,0) [rectangle,rounded corners,draw=black] {$4$};
\node (4s) at (3,1)  {$\bullet$};
\node (5) at (4,2) [rectangle,rounded corners,draw=black] {$5$};
\node (5s) at (4,1)  {$\bullet$};
\node (cdots) at (5,1)  {$\cdots$};
\node (4m-2) at (6,0) [rectangle,rounded corners,draw=black] {$4m-2$};
\node (4m-2s) at (6,1)  {$\bullet$};
\node (4m-1) at (7,2) [rectangle,rounded corners,draw=black] {$4m-1$};
\node (4m-1s) at (7,1)  {$\bullet$};
\node (4ms) at (8,1)  {$\bullet$};
\node (4m) at (9,0) [rectangle,rounded corners,draw=black] {$4m$};
\node (4m+1) at (9,2) [rectangle,rounded corners,draw=black] {$4m+1$};
\draw[very thick] (1) to (12s);
\draw[very thick] (2) to (12s);
\draw[very thick] (3) to (3s);
\draw[very thick] (4) to (4s);
\draw[very thick] (5) to (5s);
\draw[very thick] (5s) to (cdots);
\draw[very thick] (cdots) to (4m-2s);
\draw[very thick] (4m-2) to (4m-2s);
\draw[very thick] (4m-1) to (4m-1s);
\draw[very thick] (4m) to (4ms);
\draw[very thick] (4m+1) to (4ms);
\draw[very thick] (12s) to (3s);
\draw[very thick] (3s) to (4s);
\draw[very thick] (4s) to (5s);
\draw[very thick] (4m-2s) to (4m-1s);
\draw[very thick] (4m-1s) to (4ms);
\end{tikzpicture}
  \end{displaymath}
and $\tau'$  is obtained from $\tau$ by reversing the order of  the even labels, i.e. $2i$ is interchanged with $2(2m+1-i)$ for $i=1,\dots,2m$: 
  \begin{displaymath}
    \tau'=\begin{tikzpicture}[xscale=1.2,yscale=0.7, baseline =18 ]
\node (1) at (0,2) [rectangle,rounded corners,draw=black] {$1$};
\node (2) at (0,0) [rectangle,rounded corners,draw=black] {$4m$};
\node (12s) at (1,1)  {$\bullet$};
\node (3) at (2,2) [rectangle,rounded corners,draw=black] {$3$};
\node (3s) at (2,1)  {$\bullet$};
\node (4) at (3,0) [rectangle,rounded corners,draw=black] {$4m-2$};
\node (4s) at (3,1)  {$\bullet$};
\node (5) at (4,2) [rectangle,rounded corners,draw=black] {$5$};
\node (5s) at (4,1)  {$\bullet$};
\node (cdots) at (5,1)  {$\cdots$};
\node (4m-2) at (6,0) [rectangle,rounded corners,draw=black] {$4$};
\node (4m-2s) at (6,1)  {$\bullet$};
\node (4m-1) at (7,2) [rectangle,rounded corners,draw=black] {$4m-1$};
\node (4m-1s) at (7,1)  {$\bullet$};
\node (4ms) at (8,1)  {$\bullet$};
\node (4m) at (9,0) [rectangle,rounded corners,draw=black] {$2$};
\node (4m+1) at (9,2) [rectangle,rounded corners,draw=black] {$4m+1$};
\draw[very thick] (1) to (12s);
\draw[very thick] (2) to (12s);
\draw[very thick] (3) to (3s);
\draw[very thick] (4) to (4s);
\draw[very thick] (5) to (5s);
\draw[very thick] (5s) to (cdots);
\draw[very thick] (cdots) to (4m-2s);
\draw[very thick] (4m-2) to (4m-2s);
\draw[very thick] (4m-1) to (4m-1s);
\draw[very thick] (4m) to (4ms);
\draw[very thick] (4m+1) to (4ms);
\draw[very thick] (12s) to (3s);
\draw[very thick] (3s) to (4s);
\draw[very thick] (4s) to (5s);
\draw[very thick] (4m-2s) to (4m-1s);
\draw[very thick] (4m-1s) to (4ms);
\end{tikzpicture}
  \end{displaymath}

Then 
\begin{displaymath}
  \begin{array}[c]{*5c}
 D_1(\tau,\tau')&\ge& \tilde D_1(\tau,\tau')&\ge&4m^2-4m+2\\
D_2(\tau',\tau')&\ge& \tilde D_2(\tau,\tau')&\ge&\sqrt{\frac{16}3m^3-8m^2+\frac{32}3m-6}\\
&&D_\infty (\tau,\tau')&= &2m-1
  \end{array}
\end{displaymath}
\end{lemma}

\begin{proof}
  It is easy to see that for $1<i<j<n=4m+1$
  \begin{displaymath}
    \rho(i,j)=\left\{
      \begin{array}[c]{cl}
        j&i=1,2,\quad j\le 4m\\
4m&i=1,2,\quad j=4m+1\\
4m+1-i&3\le i,\quad j=4m,4m+1\\
j-i+2&\mbox{otherwise}
      \end{array}
\right.
  \end{displaymath}
First, the formula for $D_\infty(\tau,\tau')$ follows immediately from Theorem \ref{th:PDcomp}. 

Continuing,  we obtain from (\ref{eq:min2}) the following constraints
\begin{displaymath}
  \begin{array}[c]{*3c}
    \delta_1+\delta_2&\ge&4m-2\\
\delta_3+\delta_4&\ge& 4m-8\\
\delta_5+\delta_6&\ge&4m-12\\
&\vdots&\\
\delta_{2m-1}+\delta_{2m}&\ge&0\\
&\vdots&\\
\delta_{4m-3}+\delta_{4m-2}&\ge&4m-8\\
\delta_{4m-1}+\delta_{4m}&\ge&4m-4
  \end{array}
\end{displaymath}
Summing up this constraints directly gives the lower bound for $\tilde D_1(\tau,\tau')$.

Now $a^2+b^2\ge \frac{(a+b)^2}2$ gives us 
\begin{displaymath}
  \begin{array}[c]{*3c}
    \delta_1^2+\delta_2^2&\ge&2(2m-1)^2\\
\delta_3^2+\delta_4^2&\ge& 2(2m-4)^2\\
\delta_5^2+\delta_6^2&\ge&2(2m-6)^2\\
&\vdots&\\
\delta_{2m-1}^2+\delta_{2m}^2&\ge&0\\
&\vdots&\\
\delta_{4m-3}^2+\delta_{4m-2}^2&\ge&2(2m-4)^2\\
\delta_{4m-1}^2+\delta_{4m}^2&\ge&8(m-1)^2
  \end{array}
\end{displaymath}
Again summing up this yields the lower bound for $\tilde D_2(\tau,\tau')$. 
\end{proof}
\begin{remark}
  Using the results from the next section and computation similar to the second next section we could derive the same order of magnitude of $D_i$ for general $n$. 
\end{remark}
\section{Local Properties}

From Lemma \ref{lem:algebra} we obtain   for ``small'' semimetrics $\rho',\rho''\in M(X)$ immediately that
\begin{displaymath}
  D_i(\rho+\rho',\rho+\rho'')\le D_i(\rho',\rho'').
\end{displaymath}
 Notably,  we can even sharpen this estimate: 
\begin{lemma}\label{lem:localtilde}
  For all $\rho\in M_{>0}(X)$ there is a $\varepsilon>0$ such that for $\rho',\rho''\in M(X)$ with $D_\infty(0,\rho'),D_\infty(0,\rho'')<\varepsilon$
\begin{displaymath}
  D_i(\rho+\rho',\rho+\rho'')= \tilde D_i(\rho',\rho'')
\end{displaymath}
\end{lemma}
\begin{remark}
  For $\tau \in T(X)$ the condition $\rho\in M_{>0}(X)$ just means that the labeling is injective. Thus it is weaker than to say that $\tau$ is an inner point  of some orthant of tree space as considered in  \cite{BHV01}, meaning the tree is binary and all edge lengths are positive.   

Further, this result is another proof that the $\tilde D_i$ are really metrics, see Lemma \ref{lem:tildemetrics}. 
\end{remark}
In the following, let  $0\in M(X)$ denote  the zero semimetric on $X$.
\begin{proof}[Proof of Lemma \ref{lem:localtilde}]
 By Lemma \ref{lem:upperbound}, we may add the constraints $\delta_x\le 2D_\infty(\rho+\rho',\rho+\rho'')=2D(\rho',\rho'')$ to (\ref{eq:min1}) and (\ref{eq:min2}) to get problems (\ref{eq:min1b}) and (\ref{eq:min2b}), respectively. 

Now it is easy to derive  that for
\begin{displaymath}
  \varepsilon=\frac12\min\set{\rho(x,y):\set{x,y}\in \binom X2}
\end{displaymath}
and $\rho,\rho'\in M(X)$,  $D_\infty(0,\rho'),D_\infty(0,\rho'')<\varepsilon$, the constraints
\begin{displaymath}
  \abs{\delta_x-\delta_y}\le 2\rho(x,y)+\rho'(x,y)+\rho''(x,y)
\end{displaymath}
are automatically fulfilled. Removing them yields problem (\ref{eq:min2b}).  
\end{proof}

\begin{example}\label{ex:weights1tree}
  So it is interesting to ask for $\tilde D_i(0,\tau_{A,B}^l)$ for a very simple $\tau$, we choose  $\tau=\begin{tikzpicture}[xscale=0.5,yscale=0.2, baseline =1 ]
\node (A) at (0,1) [rectangle,draw=black] {$A$};
\node (B) at (2,1) [rectangle,draw=black] {$B$};
\draw[very thick] (A) to (B);
\node at (1,1) [above] {$l$};
\end{tikzpicture}$ where $A|B$ is a split of $X$ and $l$ is the length of this split.

We see that the constraints from (\ref{eq:min2}) turn into
\begin{displaymath}
    \begin{array}[c]{*4c}
 \delta_x&\ge&0&x\in X\\
    \delta_x+\delta_y&\ge&l&x \in A, y\in B
  \end{array} 
\end{displaymath}
Now  (\ref{eq:min2}) is symmetric under permutations  of $A$ and under  permutations of $B$. Thus  we may simply assume that 
\begin{displaymath}
\delta_x=\left\{
  \begin{array}[c]{*3c}
    a&x\in A\\
b&x\in B
  \end{array}
\right.
\end{displaymath}
for some $a,b\in \NRp$ with $a+b\ge l$.

For computing $\tilde D_1$, we find
\begin{displaymath}
  \norm\delta_1=\#A a+\# Bb\ge \#A a+\# B (l-a).
\end{displaymath}
The later function of $a$ has  minimum $\tilde D_1(0,\tau_{A,B}^l)=\min(\#A,\# B) l$. 

Similarly we find for $\tilde D_2$
\begin{displaymath}
  \norm\delta_2^2=\#A a^2+\# Bb^2\ge \#A a^2+\# B (l-a)^2.
\end{displaymath}
Now the minimum is $\tilde D_2(0,\tau_{A,B}^l)=\sqrt{\frac{\#A\# B}{n}} l$.

Summarisingly, we observe that  different  splits of a tree get different weights. 

Moreover, we see that the minimal points $\delta_i^*$ fulfil \emph{all} contraints in (\ref{eq:min1}). This shows  $D_i=\tilde D_i$. Further, the same computations are valid if we compute $D_i(\tau_{A,B}^l,\tau_{A,B}^{l'})$ with $\abs{l-l'}$ replacing $l$:
\begin{eqnarray*}
  D_1(\tau_{A,B}^l,\tau_{A,B}^{l'})&=&\min(\#A,\# B)\abs{l-l'}\\
  D_2(\tau_{A,B}^l,\tau_{A,B}^{l'})&=&\sqrt{\frac{\#A\# B}{n}}\abs{l-l'}\\
\end{eqnarray*}

\end{example}

\begin{example}\label{ex:ll'}
  We want to compute $\tilde D_i(0,\tau_{A,B,C}^{l,l'})$  for
  \begin{displaymath}
\tau=\begin{tikzpicture}[xscale=0.5,yscale=0.2, baseline =1 ]
\node (A) at (0,1) [rectangle,draw=black] {$A$};
\node (B) at (2,1) [rectangle,draw=black] {$B$};
\node (C) at (4,1) [rectangle,draw=black] {$C$};
\draw[very thick] (A) to (B);
\draw[very thick] (B) to (C);
\node at (1,1) [above] {$l$};
\node at (3,1) [above] {$l'$};
\end{tikzpicture}
\end{displaymath}
This tree is the essence of two trees with same shape but differing in the lengths of two edges.  

Again, symmetry gives us to consider only \begin{displaymath}
\delta_x=\left\{
  \begin{array}[c]{*3c}
    a&x\in A\\
b&x\in B\\
c&x\in C
  \end{array}
\right.
\end{displaymath}
for some $a,b,c\in \NRp$ which fulfil now
\begin{equation}\label{eq:abc}
  \begin{array}[c]{*3c}
    a+b&\ge&l\\
b+c&\ge& l'\\
a+c&\ge& l+l'
  \end{array}
\end{equation}
which gives us a linear or quadratic program in $\NRp^3$.

For computing $\tilde D_1$, we want
\begin{displaymath}
  \#Aa+\#Bb+\#Cc\mapsto\min
\end{displaymath}
on this set. We know, that this minimum is achieved in a corner of the feasible set.
 But, we see easily that not all inequalties in (\ref{eq:abc}) could be equalities unless $b=0$. Thus at least one of $a,b,c$ must be zero and we obtaine the minimal value as
 \begin{displaymath}
   \min\set{\#A l+\#Cl',(\#B+\#C)l+\#C l',\#A l+(\#B +\#C)l'}
 \end{displaymath}
A distinction of cases whether $\#A\gtreqless \#B+\#C$ and $\#C\gtreqless \#A+\#B$ gives us in any case one of the value as minimum. Thus in any case, $\tilde D_1(0,\tau_{A,B,C}^{l,l'})$ is a linear combination of $l$ and $l'$, i.e. some weighted $\ell^1-$ distance.

The computation of $\tilde D_2$ would mean solving the quadratic program 
\begin{displaymath}
  \#Aa^2+\#Bb^2+\#Cc^2\mapsto\min
\end{displaymath}
For this problem, we only know that the solution is the projection of the null vector onto the  affine hyperspace determined by some \emph{face} of the feasible set. 
This projection is linear in $l$ and $l'$. 
This means that $\tilde D_2$ is the minimum of several quadratic functions in $l,l'$.    Since the algebra is rather tedious we stop here now with the indication that this minimum is just a single quadratic function similar to the linear case before. A numerical test for several cardinalities and random lengths $l,l'$ provided in Figure \ref{fig:notquadratic} shows that the parallelogramm equality is fulfilled in all considered situations. Thus  the local geometry seems to be euclidean. This was our original expectation when we introduced $D_2$. But even if this would be  true in general,  we are already asured   by the previous example that we do not to compute the geodesic metric from \cite{BHV01}.   
\begin{figure}
  \centering
  
\includegraphics[angle=-90,width=0.7\textwidth]{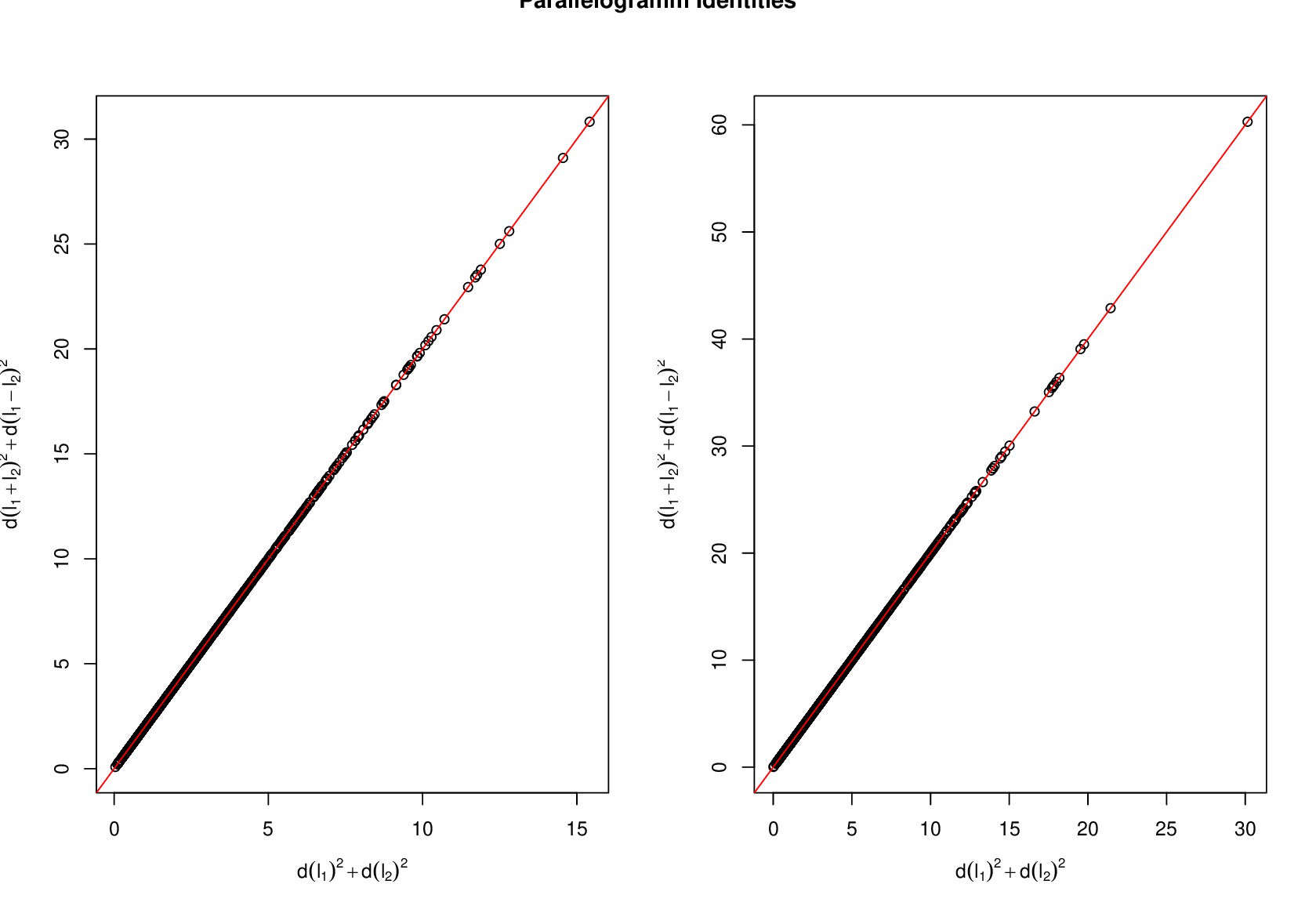}

\includegraphics[angle=-90,width=0.7\textwidth]{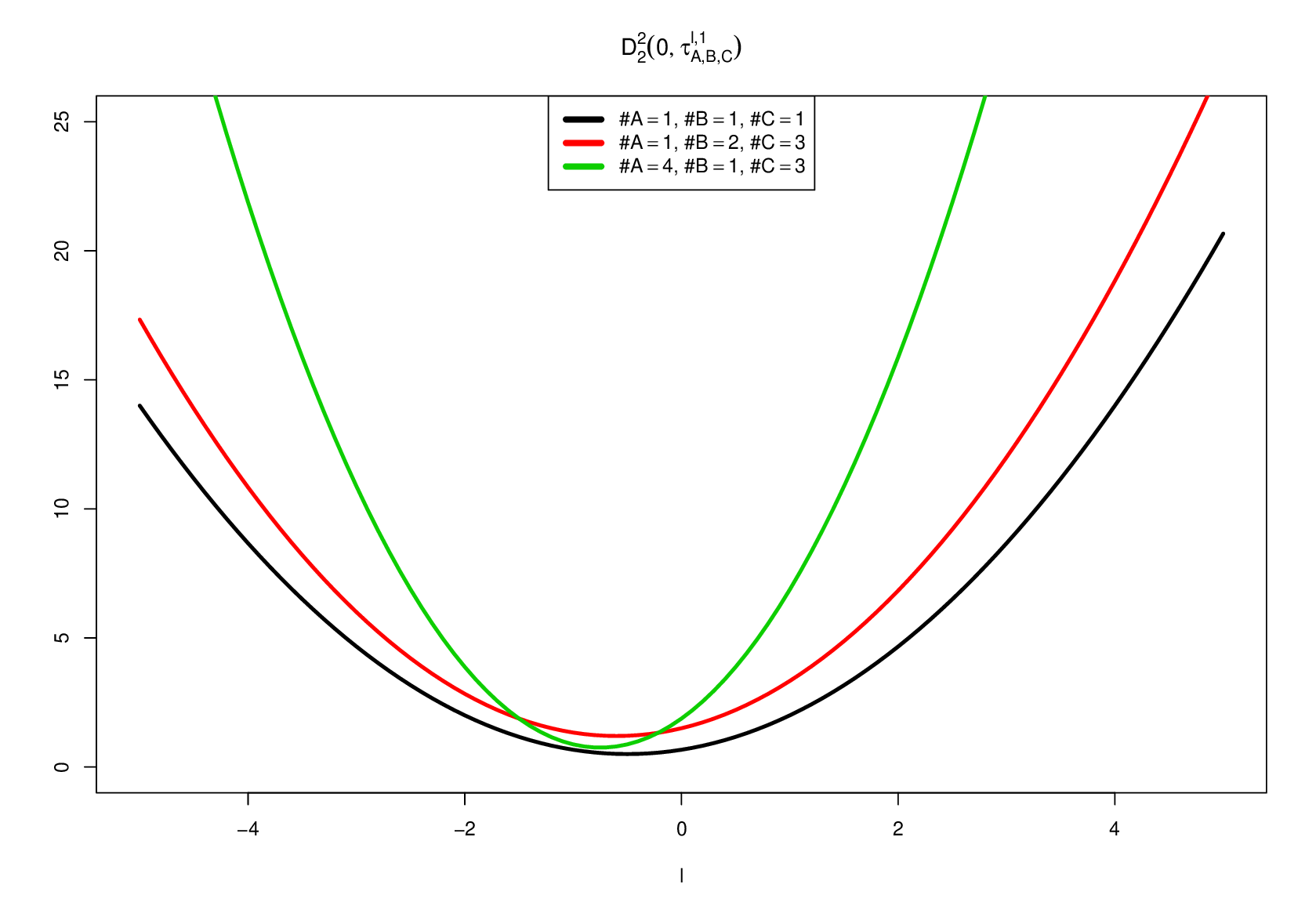}

  \caption{Test of the parallelogramm equality for random lengths $l,l'$ and $\#A=\#B=\#C=1$ (above left), $\#A=1,\#B=2,\#C=3$ (above right). On the $x-$axis $\tilde D_2(0,\tau_{A,B,C}^{l_1,l_1'})^2+\tilde D_2(0,\tau_{A,B,C}^{l_2,l'_2})^2$ is presented. On the $y-$axis $\tilde D_2(0,\tau_{A,B,C}^{l_1+l_2,l_1'+l_2'})^2+\tilde D_2(0,\tau_{A,B,C}^{l_1-l_2,l'_1-l'_2})^2$ is plotted. Below, the curves  $l\mapsto \tilde D_2(0,\tau_{A,B,C}^{l,1})^2$ for different scenarios on $\#A,\#B,\#C$ are plotted. }
  \label{fig:notquadratic}
\end{figure}
% Suppose first the projection  $a=0$. Then  the minimal value on this face is $\# B l^2+\#C (l+l')^2$.
% Similarly, we find $\#A l^2+\#C (l')^2$ on $b=0$ and $\#A (l+l')^2+\# B(l')^2$ on $c=0$.

% For the rest of the discussion of cases, the minimal point should fulfil $a,b,c>0$. Thus regarding  (\ref{eq:abc}), not in all three inequalities could  equality hold. 

% Consider first $a+b=l$. We obtain    
% \begin{displaymath}
%   \#Aa^2+\#Bb^2+\#Cc^2= \#Aa^2+\#B(l-a)^2+\#Cc^2
% \end{displaymath}
% under $c-a\ge l'-l$ and $a+c\ge l+l'$. Still we have to look for a face of that set for a minimal point. We conclude that we need two equalities anyway. 
% $a+b=l$ and $b+c=l'$  gives $b=0$ and is excluded.
%  $a+b=l$ and $a+c=l'+l$  gives
% \begin{displaymath}
%   \#Aa^2+\#Bb^2+\#Cc=\#A a^2+\#B (l-a)^2+\#C (l'+l-a)^2
% \end{displaymath}
% under $a\le l$. But $a=l$ is impossible due to $b=0$ and we derive $a=\frac{\#Bl+\#C(l+l')}n$ as one possibility if it is feasible. 

% We see that the structure of optimal values is much more complicated now .
% Lagrangian multipliers:
% \begin{eqnarray*}
%   2\#Aa-\mu_{ab}-\mu_{ac}&=&0\\
%   2\#Bb-\mu_{ab}-\mu_{bc}&=&0\\
%   2\#Cc-\mu_{ac}-\mu_{ac}&=&0\\
% \mu_{ab},\mu_{ac},\mu_{bc}&\ge&0\\
%     a+b&\ge&l\\
% b+c&\ge& l'\\
% a+c&\ge& l+l'\\
% \mu_{ab}(l-a-b)&=&0\\
% \mu_{ac}(l+l'-a-c)&=&0\\
% \mu_{bc}(l'-b-c)&=&0
% \end{eqnarray*}
% The first three equalities yield:

\end{example}

\section{Monotony}\label{sec:monoton}
For any $X_0$-tree $\tau$ let $\tau|_X$ denote the restriction to $X\subseteq X_0$. Observe that for $\tau\in T_1(X_0)$ in general  $\tau|_X\notin T_1(X)$. 
\begin{lemma}
  Let $X_0\supseteq X$ and $\tau,\tau'\in T(X_0)$. Then for $i=1,2,\infty$ 
  \begin{displaymath}
    \begin{array}[c]{*3c}
    D_i(\tau,\tau')&\ge& D_i(\tau|_X,\tau'|_X)\\
    \tilde D_i(\tau,\tau')&\ge& \tilde D_i(\tau|_X,\tau'|_X)
  \end{array}\end{displaymath}
\end{lemma}
\begin{proof}
  This follows immediately from the same inequalities for semimetrics on $X_0$. Then, restricting $d^*_i\in E(\rho,\rho')$ from Lemma \ref{lem:di*} to $X\cup X'$ yields an element of $E(\rho|_X,\rho'|_X)$. Moreover, 
  \begin{displaymath}
    \norm{(\delta_x)_{x\in X_0}}_i\ge  \norm{(\delta_x)_{x\in X}}_i
  \end{displaymath}
for $\delta\in \NRp^{X_0}$ completes the calculation. 
\end{proof}
\begin{remark}
  This  result naturally holds for many other phylogenetic metrics:  for the pathwise difference, NNI-, SPR-, TBR- and maximum parsimony metrics,  for example. For the tree rearrangement metrics is was shown in \cite[Lemma 2.2]{AS01}.
\end{remark}

\section{Implementation and numerical examples}
\label{sec:numeric}
The different metrics were implemented  by   \texttt{R} \cite{R-source} programs.  For solving  linear and quadratic  programs  the \texttt{glpkAPI} library \cite{glpkAPI} and \texttt{quadprog} library \cite{quadprog} were used, respectively. The corresponding \texttt{R}-script can be downloaded from the website \cite{gromovsource}.
Some testing showed best performance in terms of computing time for the dual simplex algorithm in the $\ell^1$-case. The computing time for obtaining the distance between  random trees of size 100 was around 0.3s which is quite reasonable, see Figure \ref{fig:times}. It also compares with the computing time of the geodesic distance.  The random trees were generated by the function \texttt{rtree} of the \texttt{R} library \texttt{phangorn} \cite{phan}.
\begin{figure}
  \centering

  \includegraphics[angle=-90,width=0.6\textwidth]{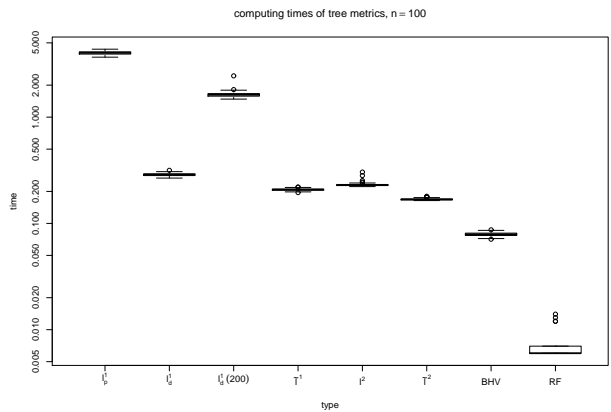}
  \caption{Computing times different metrics (logarithmic scale) for random trees with $n=100$ using the dual simplex algorithm.  From left:  $D_1$ but with primal simplex algorithm, $D_1$,  $D_1$ for $n=200$,  $\tilde D_1$, $D_2$, $\tilde D_2$, the geodesic and the Robinson-Foulds metric.}
  \label{fig:times}
\end{figure}

We also compared $D_i$ and $\tilde D_i$ with several other phylogenetic  metrics, essentially the pathwise difference, the geodesic distance and the Robinson-Foulds metric, for $n=10$ leaves. For the computation of the geodesic (BHV-) metric the \texttt{R}-package \texttt{distory} \cite{distory} was used.  The results are presented in Figure \ref{fig:distcont}. Numerially, we could observe $D_i=\tilde D_i$ in all cases, seee Figure \ref{fig:gleich} at the end of the paper.   A remarkable correlation between the different Gromov-type and the pathwise difference metrics can be observed. There is not much correlation to the geodesic distance. May be, the different weigths on the internal edges (see example \ref{ex:weights1tree}) are responsible for that. 

Similar pictures are found for unweighted trees, see Figure \ref{fig:distdiscr}. Interestingly, $D_1=\tilde D_1$ turns out to integer-valued now, see the same figure. That is quite a bit surprising since the matrix corresponding to the linear program (\ref{eq:min2}) is not totally unimodular in the sense of \cite{Ber}, it contains the $3\times3$  submatrix $
\begin{pmatrix}
  1&1&0\\1&0&1\\0&1&1
\end{pmatrix}$ with determinant $-2$.

Random caterpillars are  interesting in their own, the results are presented in Figure \ref{fig:distcaterpillar}. We observe that we obtain a much larger maximum of 28 for  $D_1$  (over the sample) than from random trees. In comparison, the  lower bound from   Lemma \ref{lem:cater} would be much smaller: $\frac{n^2}4-n+2=17$.

\begin{figure}
  \centering

  \includegraphics[angle=-90,width=0.6\textwidth]{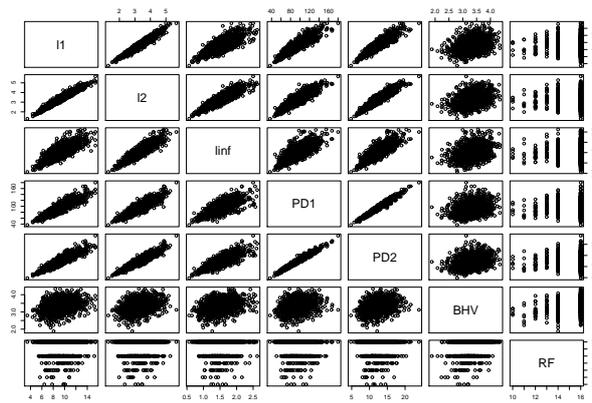}

 \includegraphics[angle=-90,width=0.6\textwidth]{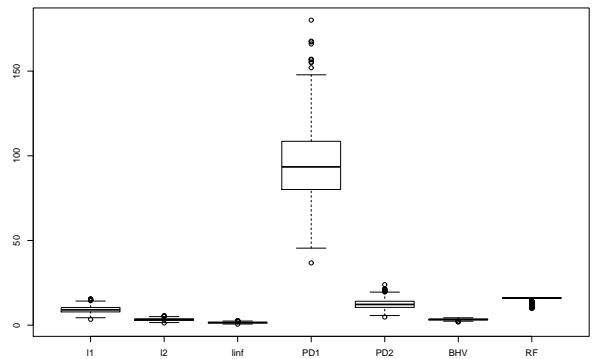}
  \caption{Comparison of  different metrics  for random trees with $n=10$. Above from upper left:  $D_1$ , $D_2$, $ D_\infty$, $D^{PD}_1$, $D_2^{PD}$, the geodesic and  the Robinson-Foulds metric. Below, the distributions are presented in boxplots. }
  \label{fig:distcont}
\end{figure}
\begin{figure}
  \centering

  \includegraphics[angle=-90,width=0.6\textwidth]{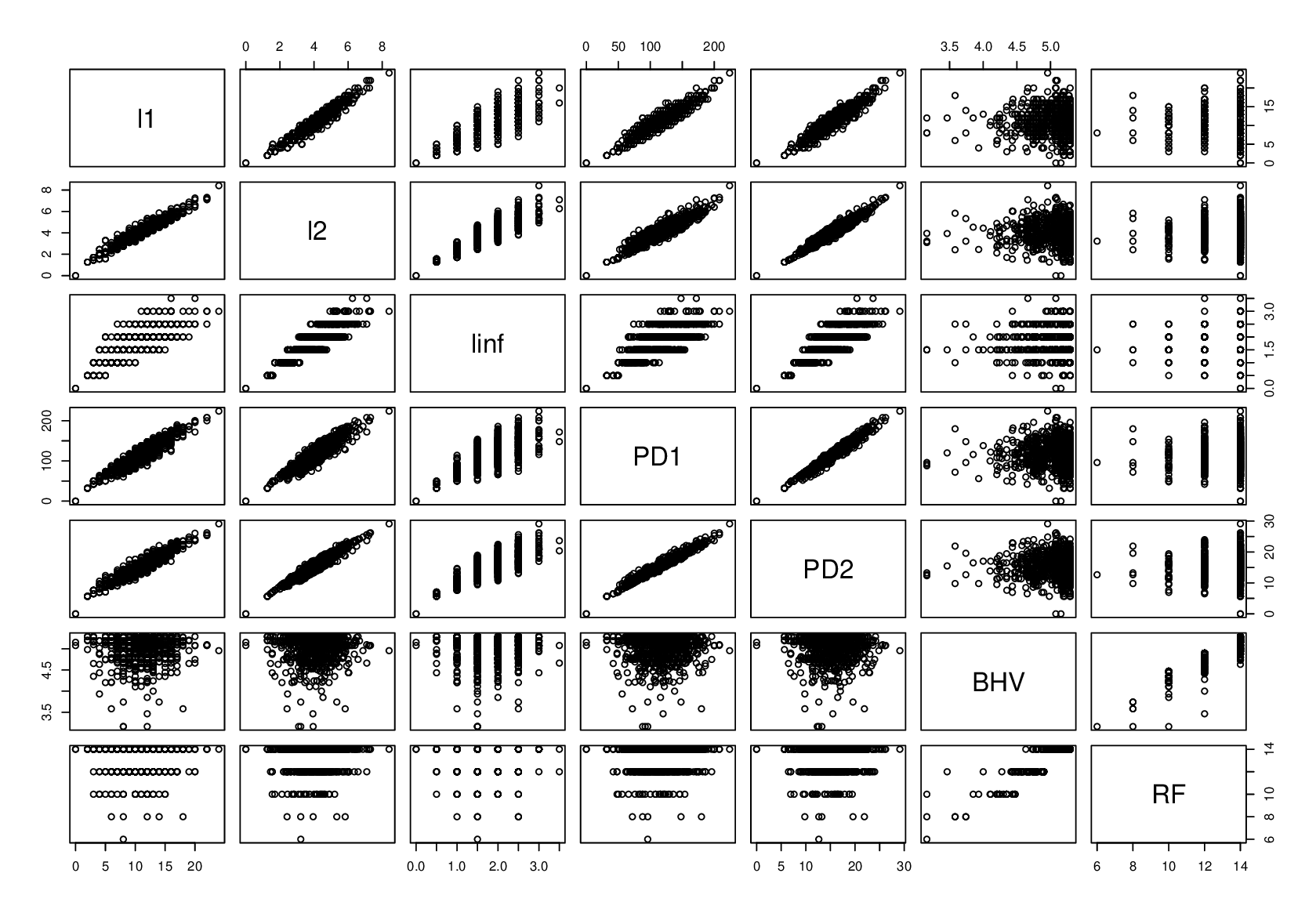}

  \includegraphics[angle=-90,width=0.6\textwidth]{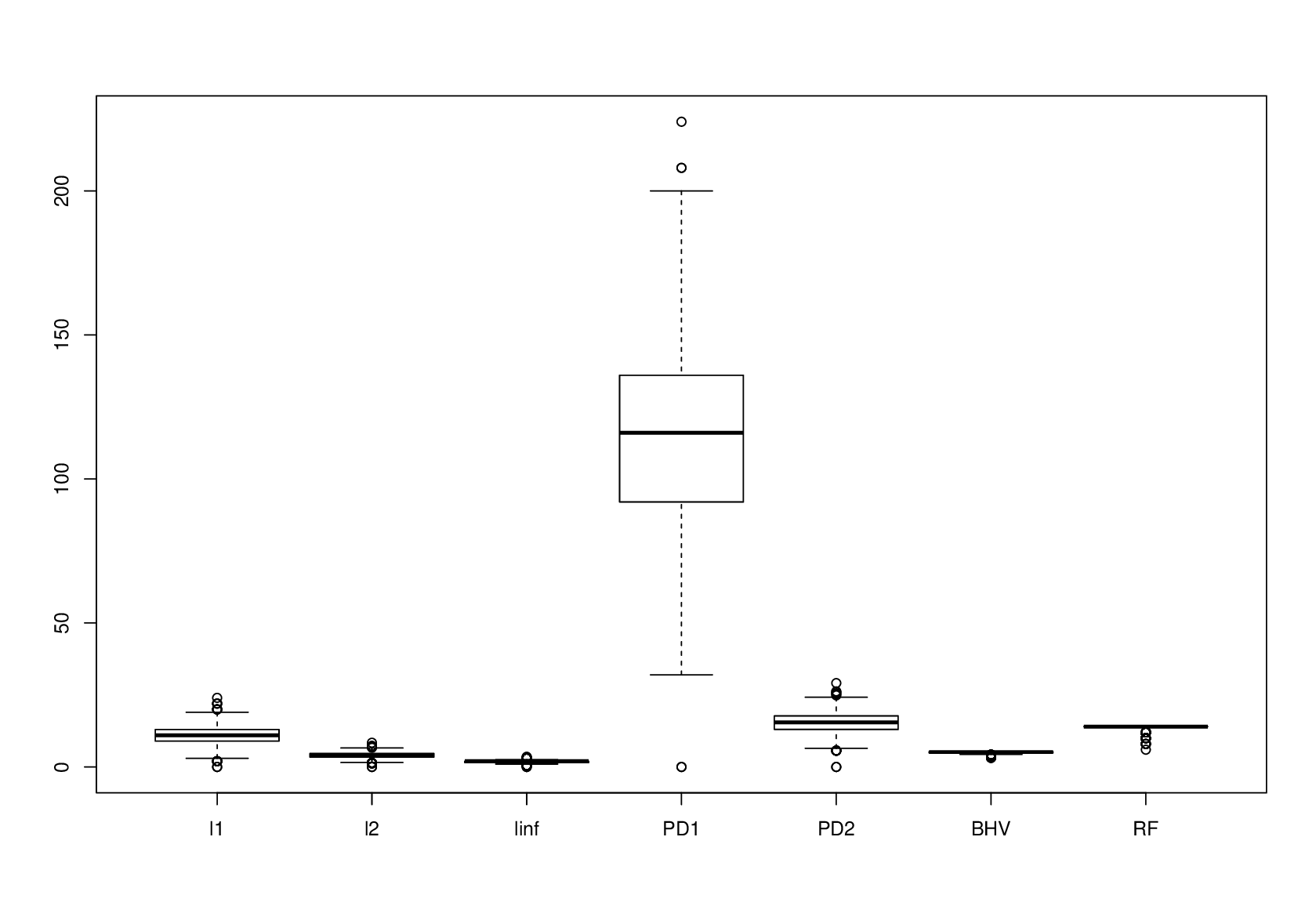}

  \includegraphics[angle=-90,width=0.6\textwidth]{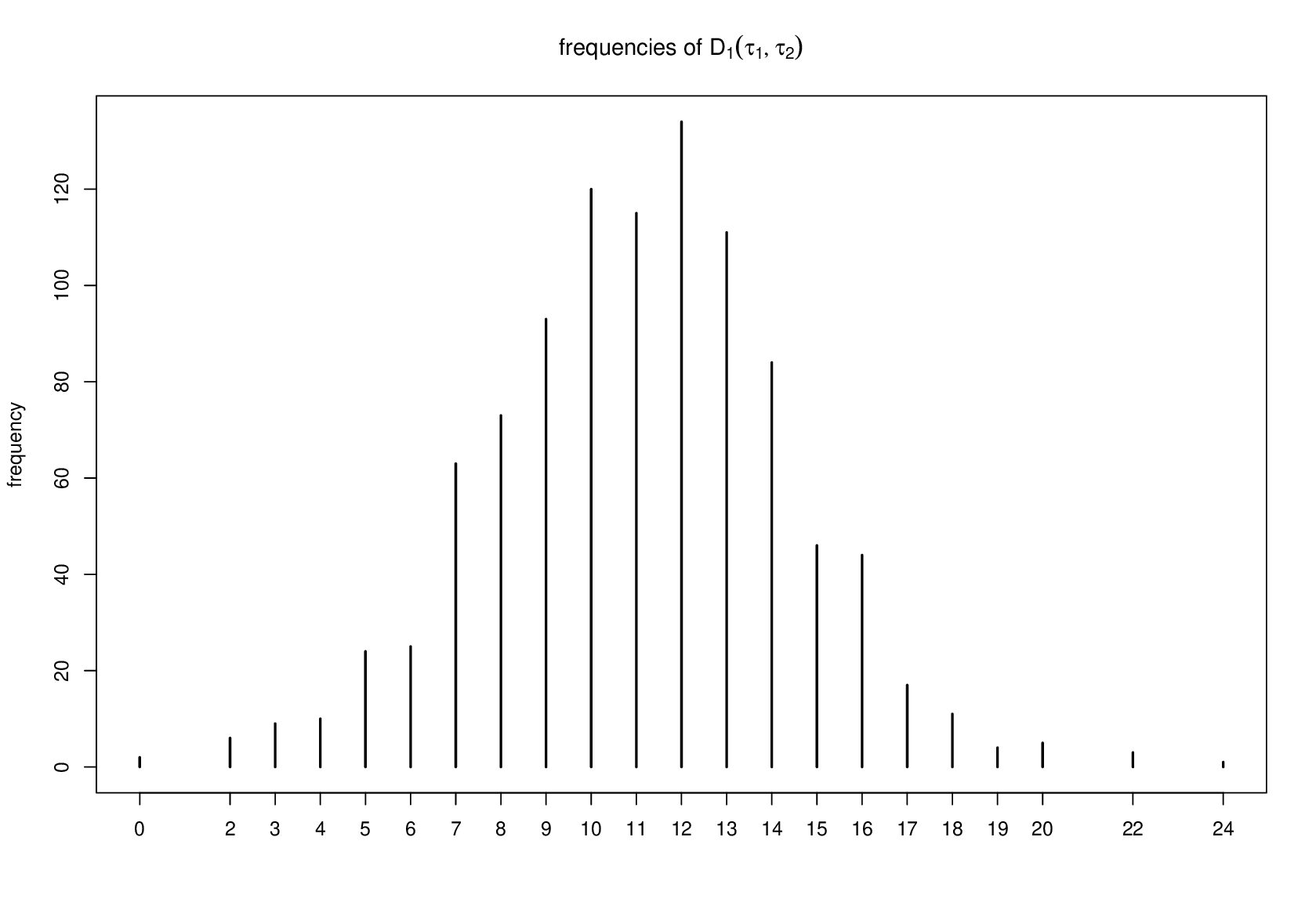}

  \caption{Comparison of  different metrics  for random  unweighted trees with $n=10$. Above from upper left:  $D_1$ , $D_2$, $ D_\infty$, $D^{PD}_1$, $D_2^{PD}$, the geodesic and  the Robinson-Foulds metric. In the middle, the distributions are presented in boxplots. At the bottom, the frequency table of $D_1$ is presented. }
  \label{fig:distdiscr}
\end{figure}
\begin{figure}
  \centering

  \includegraphics[angle=-90,width=0.6\textwidth]{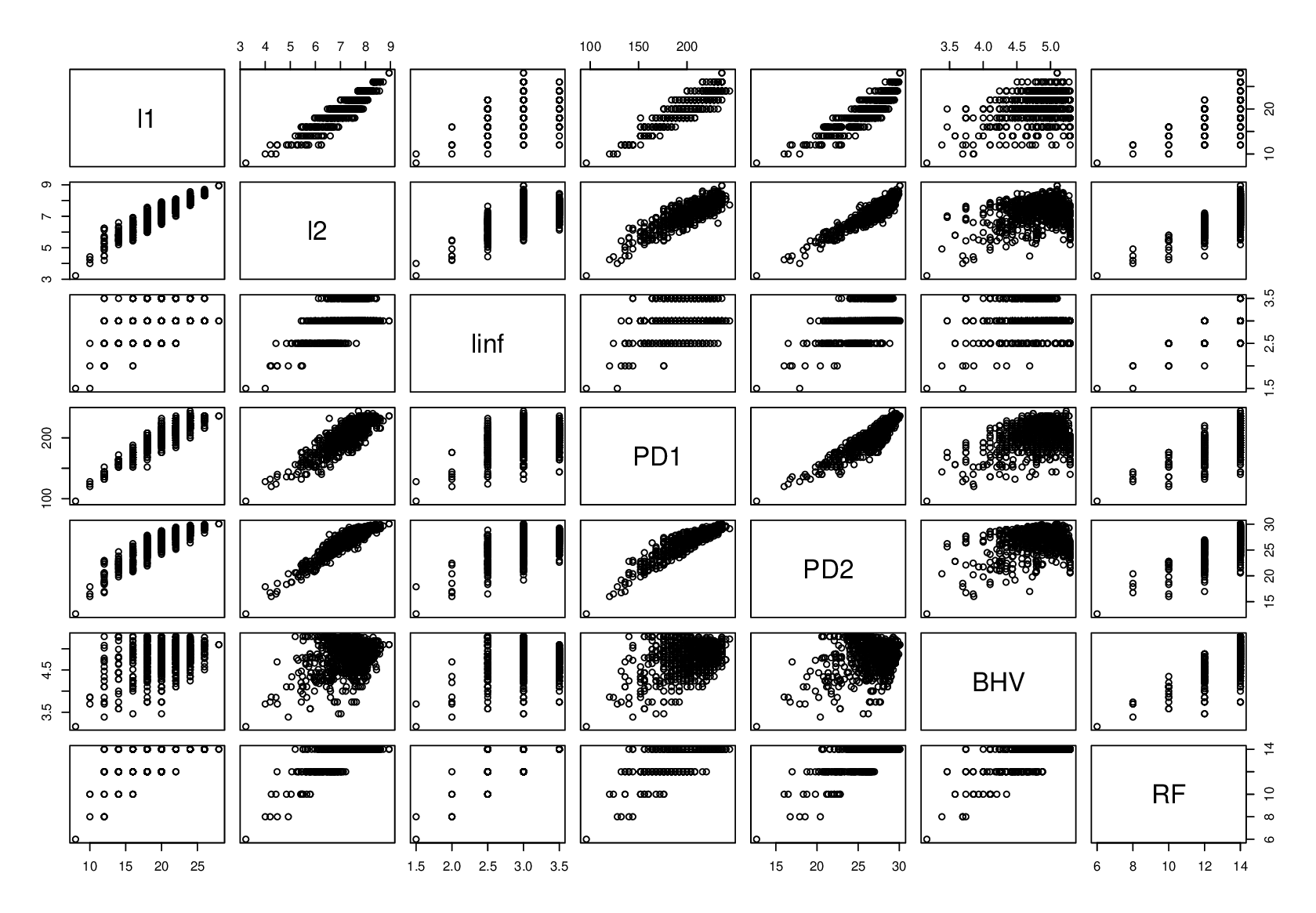}

  \includegraphics[angle=-90,width=0.6\textwidth]{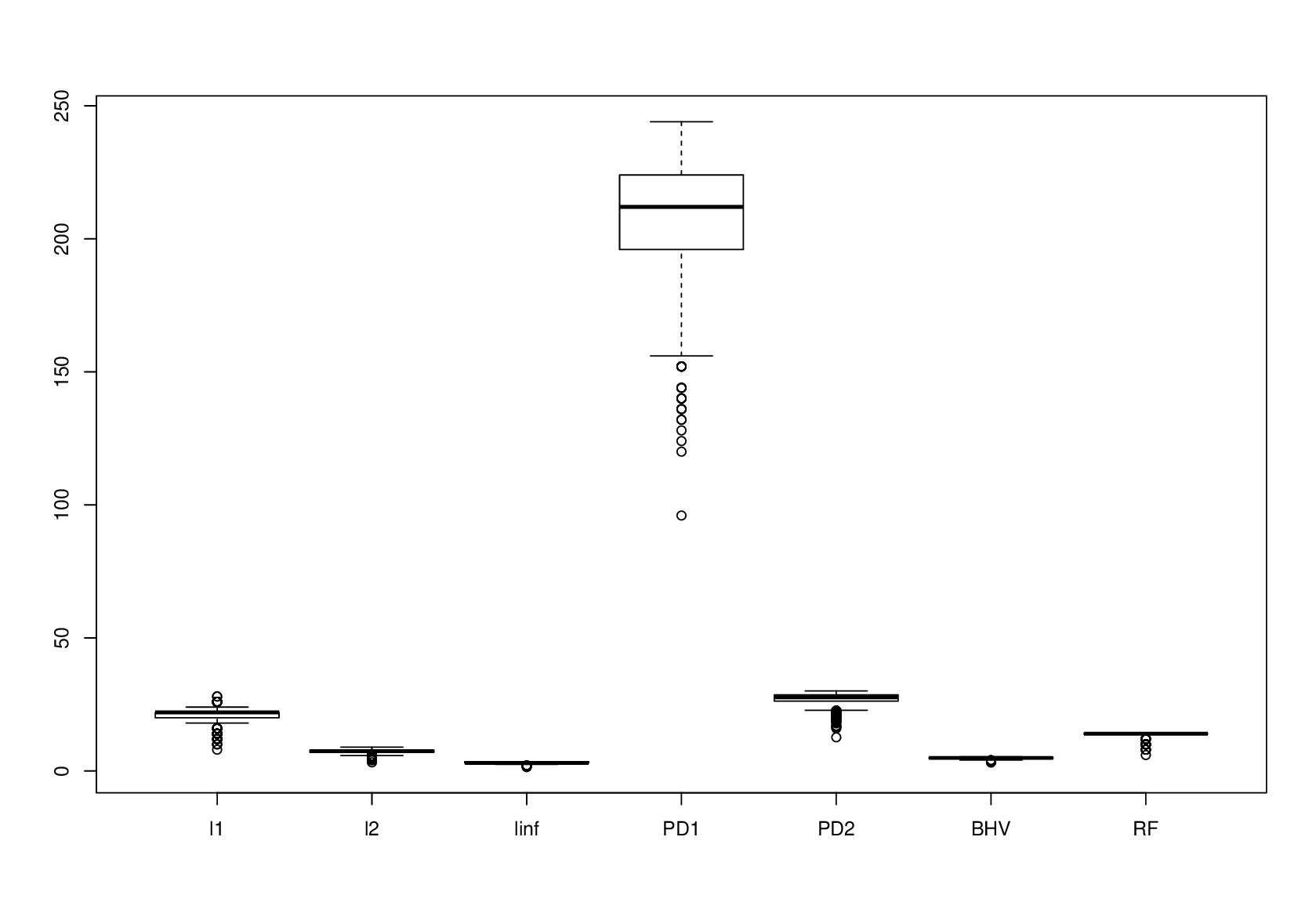}

  \includegraphics[angle=-90,width=0.6\textwidth]{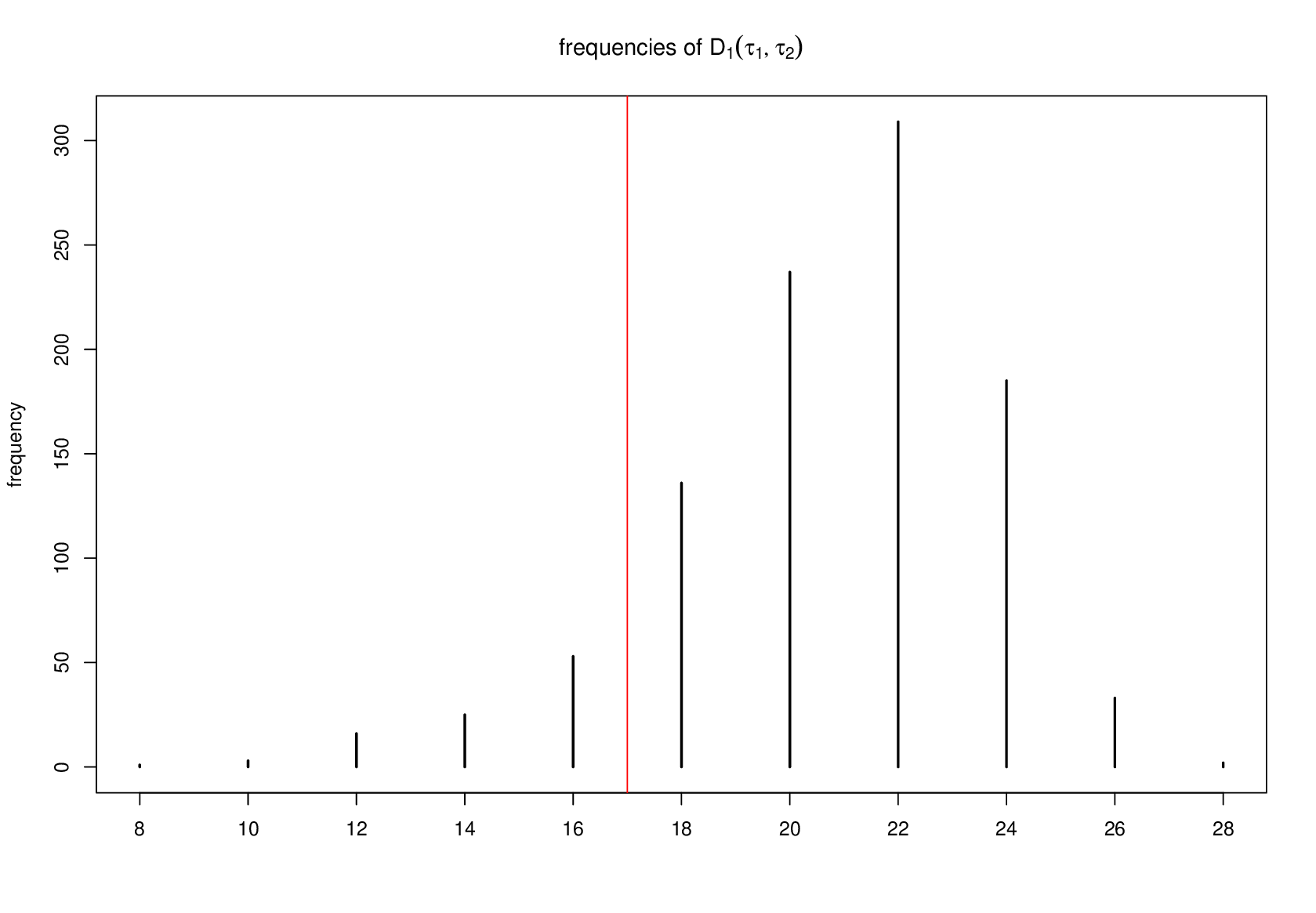}

  \caption{Comparison of  different metrics  for random  caterpillars with $n=10$. Above from upper left:  $D_1$ , $D_2$, $ D_\infty$, $D^{PD}_1$, $D_2^{PD}$, the geodesic and  the Robinson-Foulds metric. In the middle, the distributions are presented in boxplots. At the bottom, the frequency table of $D_1$ is presented with the lower bound from Lemma \ref{lem:cater} in red. }
  \label{fig:distcaterpillar}
\end{figure}
\section{Discussion}

What have we achieved? We constructed at least two new biocomputable metrics for comparing unrooted, but possibly weighted, phylogenetic trees. We think this approach is valuable and could generalise well.   One direction is the extension to rooted trees. We should then just measure the distance of the induced metrics on $X\cup\set{root}$. Another generalisation could be phylogenetic networks. Outside phylogeny, there should be applications to other  kinds of finite labeled metric spaces. At the moment,  we are only  aware of the papers of F.Memoli, e.g.  \cite{Mem07}, which deals with $\ell^p-$type Gromov-Hausdorff  metrics.  

In general, we follow \cite{SP93} in arguing that there is no universal  metric for phylogenetic trees which suits perfectly for all purposes. We think that every application has its own choice, and we added a further choice to this portfolio. Yet, we should discuss further  properties of phylogenetic metrics to guide the users. Monotony as considered in section \ref{sec:monoton} is a, yet trivial, beginning in this direction.  Here  we   want to discuss  some important results of the present paper and possible extensions only.

It looks interesting to extend  the metric to tree shapes, with allowing the labels to be permuted. Still that metric differs from the Gromov-Hausdorff metric since we allow only matching of the labels in contrast to the weaker version in (\ref{eq:Gromovhere}).  For the Gromov-Hausdorff distance it is shown in   \cite{PW94} that it is again NP-hard to compute it.   We expect the same for the permutation approach. 

One important topic which raised up already in \cite{BM12,LRM12,DG15,Ken15} is the question how to \emph{weight} the edges of the trees. We obtained natural weights from our approach in Example \ref{ex:weights1tree}. If those weights do not fit the intention of the applicant, it is easy to shorten or lengthen the edges of the trees and obtain other metric spaces which could be easily compared.  There is also the possibility to weight the labels, for instance to account for uneven sampling. Then we could adjust to this by weighting the $\norm\cdot_i$ norms which leads again to similar computations. Note that we met already such a weighted approach in the computations in the Examples \ref{ex:weights1tree} and \ref{ex:ll'}.    Further, also a Kantorovich-Wasserstein approach similar to \cite{Mem07} might be feasible if the weights of the leaves differ between the trees.    In summary, our approach is natural but can be well adjusted to the needs of applications.    

We showed several properties of the new metrics including compatibility with the NNI-metric, a lot of estimates with the pathwise difference metrics, local properties related to the lower bound metrics $\tilde D_i$, and monotony.  Of course, there are many more questions in this context. Especially we would like to sharpen the  estimates. We do not know much about the $1-$neighbourhoods on $T_1^2(X)$, e.g. whether there are islands in the sense of \cite{BM12}.  There are a lot of connections with the quartet, SPR-,TBR-, maximum parsimony, weighted matching   and BHV-metrics to explore, too. Numerical comparison was done for the \texttt{R}-implemented distances only. 

 We expect the  diameter between two unweighted $X$-trees to be realised by  caterpillar trees. The simulation result in Figure \ref{fig:distcaterpillar} points into this direction. A more sharp estimate than provided in  Lemma \ref{lem:upperdiameter} and Lemma \ref{lem:cater} would be quite interesting, too. It is still not clear whether and  why $D_1$ or $\tilde D_1$ takes integers values only on $T_1^2(X)$.  

The geometry induced by the euclidean type metrics $D_2,\tilde D_2$ should be further explored, too. It should  be interesting to prove it is locally euclidean and to find out how the geodesics look like. Possibly, the geodesic distance with respect to  $D_2$ is even another metric.

Most interesting we find the question whether $D_i=\tilde D_i$. Provable equality could  save some computing time, at least.  For the time until this problem is solved, we just know  there are new  animals in the zoo of phylogenetic distances \dots  but not,   how many.  

\paragraph{Acknowledgements} First of all, I have to thank  Mareike Fischer for introducing me to the world of phylogenetic distances. She helped also a lot for getting a clear notation.  Second, I'm very grateful to Jürgen Eichhorn who unconsciously draw my attention to metrics between metric spaces. Third, I'd like to thank Michelle Kendall for her inspiring talk at the Portobello conference 2015 and additional discussion later. Fourth, I thank Mike Steel for many interesting discussions, useful hints, his kind hospitality during my stay in Christchurch 2010,  and for the organisation of the amazing 2015 workshop in Kaikoura with an inspiring and open atmosphere. Further,  Andrew Francis, Alexander Gavryushkin, Stefan Grünewald, Marc Hellmuth and  Giulio dalla Riva  gave useful hints and inspiration in many  discussions.

\appendix
\section{On metric extensions}
\label{sec:metrext}
Several times we met the problem whether  a partial dissimilarity on $X$, i.e. a map  $\map qE\NRp$, $E\subseteq\binom X2$ has an extension to a metric on $X$. This seems to be a well-known problem, one folklore solution I found in \cite{Gue04}:

\begin{theorem}\label{th:metext}
  If the graph $G=(X,E)$ is simple and connected then $\map qE\NRp$ extends to a semimetric on $X$ if and only if for all $x,y\in X$, $\set{x,y}\in E$,  $q(x,y)=d^q_G(x,y)$. 
\end{theorem}
The graph metric $d^q_G$ was introduced in (\ref{eq:dqG}).

Although this presents a complete solution of the extension problem  we want to sharpen this criterion for improved   applicability. Still the next result should be folklore but I could not find it in literature.   If $p=x_0x_1\dots x_m$ is a cycle in  a graph $(X,E)$ we call any pair $\set{x_i,x_j}\in E$, $0\le i,j\le m-1, 2\le\abs{i-j}\le m-2 $ a chord of $p$. A cycle $p$ without chord is called minimal cycle. 

\begin{theorem}\label{th:metextmin}
  If the graph $G=(X,E)$ is simple and connected, then $\map qE\NRp$ extends to a metric on $X$ if and only if for all minimal cycles $p$ of $G$ and all edges $\set{x,y}$ in $p$
  \begin{equation}
    \label{eq:boundedgelength}
    2q(x,y)\le \mathrm{len}(p).   
  \end{equation}

\end{theorem}
\begin{proof}
We assume the opposite. Thus we find a (non-minimal) cycle $p=x_0x_1,\dots x_m=x_0$ such that $e=\set{x_0,x_1}$  violates (\ref{eq:boundedgelength}). We may assume w.l.o.g. that the length of $p$, $m$ is minimal.  

Non-minimality of $p$ implies that there is a chord $\set{x_i,x_j}$ of $p$. Since $m$ is minimal,
we know
\begin{displaymath}
d(x_i,x_j)\le\sum_{k=i}^{j-1}d(x_k,x_{k+1})
\end{displaymath}
 and 
 \begin{displaymath}
d(x_0,x_1)\le \sum_{k=1}^{i-1}d(x_k,x_{k+1})+d(x_i,x_j)+\sum_{k=j}^{n-1}d(x_k,x_{k+1})
\end{displaymath}
Substituting the first inequality into the RHS of the second one yields 
\begin{displaymath}
  d(x_0,x_1)\le \sum_{k=1}^{n-1}d(x_k,x_{k+1}).
\end{displaymath}
This  is (\ref{eq:boundedgelength}). This contradiction completes the proof. 
\end{proof}

We can use this result for the  
\begin{proof}[Proof of Theorem \ref{th:quadrin}]
 We are using Theorem \ref{th:metextmin} below on $X\cup X'$ with $E=\binom X2\cup \binom{X'}2\cup\set{\set{x,x'}:x\in X}$. The minimal cycles in $(X\cup X')$ are either  triangles in $X$, triangles in $X'$ or rectangles $x,y,y',x'$. For the two former, (\ref{eq:boundedgelength}) is equivalent to  the triangle inequalities for $\rho,\rho'$. For the latter, (\ref{eq:boundedgelength}) is the same as  (\ref{eq:quadr}).
\end{proof}

The following result was used in the proof of Theorem \ref{th:metrics}.
\begin{lemma}\label{lem:biextension}
  Suppose $X,Y,Z$ are  disjoint sets and there are given $d_1\in M(X\cup Y)$ and $d_2\in M(Y\cup Z)$ such that $d_1|_{\binom Y2}=d_2|_{\binom Y2}$. Then there exists a $d\in M(X\cup Y\cup Z)$ such that $d|_{\binom {X\cup Y}2}=d_1$ and $d|_{\binom {Y\cup Z}2}=d_2$.
\end{lemma}
\begin{proof}
  Now we apply  the theorem to the graph $(X\cup Y\cup Z,\binom {X\cup Y}2\cup\binom {Y\cup Z}2)$ with $w(u,v)=\left\{
    \begin{array}[c]{cl}
d_1(u,v)&u,v\in X\cup Y\\
d_2(u,v)&u,v\in Y\cup Z
\end{array}\right.
$. Since both $X\cup Y$ and $Y\cup Z$ are cliques in this graph, the only minimal cycles are triangles. For them (\ref{eq:boundedgelength}) is fulfilled by definition of $w$. 
\end{proof}

\begin{figure}[b]
  \centering

  \includegraphics[angle=-90,width=0.6\textwidth]{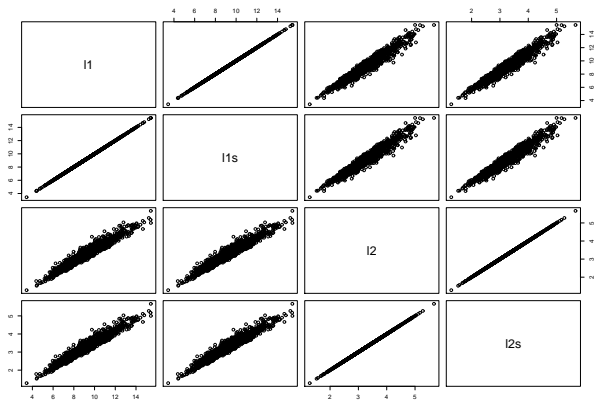}
  \caption{Equality of $D_i$ with $\tilde D_i$, $i=1,2$ for random trees with $n=10$}
  \label{fig:gleich}
\end{figure}

\end{document}